\documentclass{amsart}

\usepackage{array}
\usepackage{amssymb,mathrsfs,graphicx, enumerate}

 \usepackage{hyperref}
\usepackage{letltxmacro}
\LetLtxMacro{\oldsqrt}{\sqrt}
\renewcommand{\sqrt}[2][]{\,\oldsqrt[#1]{#2}\,}

\usepackage{bbm}
\makeatletter

\def\greekbolds#1{%
 \@for\next:=#1\do{%
    \def\X##1;{%
     \expandafter\def\csname
     V##1\endcsname{\boldsymbol{\csname##1\endcsname}} 
     }
   \expandafter\X\next;
  }
}

\greekbolds{alpha,beta,iota,gamma,lambda,nu,eta,Gamma,varsigma,Lambda}

\def\make@bb#1{\expandafter\def
  \csname bb#1\endcsname{{\mathbb{#1}}}\ignorespaces}

\def\make@bbm#1{\expandafter\def
  \csname bb#1\endcsname{{\mathbbm{#1}}}\ignorespaces}

\def\make@bf#1{\expandafter\def\csname bf#1\endcsname{{\bf
      #1}}\ignorespaces} 

\def\make@gr#1{\expandafter\def
  \csname gr#1\endcsname{{\mathfrak{#1}}}\ignorespaces}

\def\make@scr#1{\expandafter\def
  \csname scr#1\endcsname{{\mathscr{#1}}}\ignorespaces}

\def\make@cal#1{\expandafter\def\csname cal#1\endcsname{{\mathcal
      #1}}\ignorespaces} 

\def\do@Letters#1{#1A #1B #1C #1D #1E #1F #1G #1H #1I #1J #1K #1L #1M
                 #1N #1O #1P #1Q #1R #1S #1T #1U #1V #1W #1X #1Y #1Z}
\def\do@letters#1{#1a #1b #1c #1d #1e #1f #1g #1h #1i #1j #1k #1l #1m
                 #1n #1o #1p #1q #1r #1s #1t #1u #1v #1w #1x #1y #1z}
\do@Letters\make@bb   \do@letters\make@bbm
\do@Letters\make@cal  
\do@Letters\make@scr 
\do@Letters\make@bf \do@letters\make@bf   
\do@Letters\make@gr   \do@letters\make@gr
\makeatother

\newcommand{\Lsymb}[2]{\genfrac{(}{)}{}{}{#1}{#2}}
\newcommand{\ulm}{{\underline{m}}}
\newcommand{\uln}{{\underline{n}}}
\newcommand{\brN}{\breve{\mathbb{N}}}
\newcommand{\abs}[1]{\lvert #1 \rvert}

\DeclareMathSymbol{\twoheadrightarrow} {\mathrel}{AMSa}{"10}

\DeclareMathOperator{\ord}{ord}

\DeclareMathOperator{\Res}{Res}

\DeclareMathOperator{\Pic}{Pic}

\DeclareMathOperator{\Aut}{Aut}
\DeclareMathOperator{\End}{End}
\DeclareMathOperator{\Hom}{Hom}

\DeclareMathOperator{\Gal}{Gal}
\DeclareMathOperator{\Mat}{Mat}

\DeclareMathOperator{\Nm}{N}  

\DeclareMathOperator{\GL}{GL}
\DeclareMathOperator{\SL}{SL}

\newcommand{\wh}{\widehat}
\DeclareMathOperator{\Cl}{Cl}

\DeclareMathOperator{\Nr}{Nr}

\def\makeop#1{\expandafter\def\csname#1\endcsname
  {\mathop{\rm #1}\nolimits}\ignorespaces}
\makeop{Hom}   \makeop{End}   \makeop{Aut}   \makeop{Isom}  \makeop{Pic} 
\makeop{Gal}   \makeop{ord}   \makeop{Char}  \makeop{Div}   \makeop{Lie} 
\makeop{PGL}   \makeop{Corr}  \makeop{PSL}   \makeop{sgn}   \makeop{Spf}
\makeop{Spec}  \makeop{Tr}    \makeop{Nr}    \makeop{Fr}    \makeop{disc}
\makeop{Proj}  \makeop{supp}  \makeop{ker}   \makeop{im}    \makeop{dom}
\makeop{coker} \makeop{Stab}  \makeop{SO}    \makeop{SL}    \makeop{SL}
\makeop{Cl}    \makeop{cond}  \makeop{Br}    \makeop{inv}   \makeop{rank}
\makeop{id}    \makeop{Fil}   \makeop{Frac}  \makeop{GL}    \makeop{SU}
\makeop{Nrd}   \makeop{Sp}    \makeop{Tr}    \makeop{Trd}   \makeop{diag}
\makeop{Res}   \makeop{ind}   \makeop{depth} \makeop{Tr}    \makeop{st}
\makeop{Ad}    \makeop{Int}   \makeop{tr}    \makeop{Sym}   \makeop{can}
\makeop{length}\makeop{SO}    \makeop{torsion} \makeop{GSp} \makeop{Ker}
\makeop{Adm}   \makeop{Mat}

\def\opp{\mathrm{opp}}
\def\ul{\underline}

\def\wh{\widehat}

\newcommand{\isoto}{\stackrel{\sim}{\longrightarrow}}
\newcommand{\embed}{\hookrightarrow}

\def\Fqbar{\overline{\bbF}_q}
\def\Fpbar{\overline{\bbF}_p}
\def\Fp{{\bbF}_p}
\def\Fq{{\bbF}_q}

\newcommand{\Z}{\mathbb Z}
\newcommand{\Q}{\mathbb Q}
\newcommand{\R}{\mathbb R}
\newcommand{\C}{\mathbb C}
\newcommand{\G}{\mathbb G}
\newcommand{\F}{\mathbb F}
\renewcommand{\O}{\mathbb O}  

\newcounter{thmcounter} 
\numberwithin{thmcounter}{section}
\newtheorem{thm}[thmcounter]{Theorem}
\newtheorem{lem}[thmcounter]{Lemma}

\newtheorem{cor}[thmcounter]{Corollary}
\newtheorem{prop}[thmcounter]{Proposition}
\theoremstyle{definition}

\newtheorem{rem}[thmcounter]{Remark}

\numberwithin{equation}{section}

\newtheoremstyle{notitle}  
  {}
  {}
  {\itshape}
  {}
  {}
  {\ }
  {.5em}
  {}
\theoremstyle{notitle}

\title[Superspecial abelian surfaces]{On superspecial abelian surfaces
  over finite fields II}
\author{Jiangwei Xue}

\address{(Xue) Collaborative Innovation Centre of Mathematics, 
School of Mathematics and Statistics, Wuhan University, Luojiashan,
Wuhan, Hubei, 430072, P.R. China.}

\address{(Xue) Hubei Key Laboratory of Computational Science (Wuhan
  University), Wuhan, Hubei,  430072, P.R. China.}
\email{xue\_j@whu.edu.cn}

\author{Tse-Chung Yang}
\address{(Yang) Institute of Mathematics, Academia Sinica,
  Astronomy-Mathematics Building, 6F, No. 1, Sec. 4, Roosevelt Road,
  Taipei 10617, TAIWAN.} 
\email{tsechung@math.sinica.edu.tw}

\author{Chia-Fu Yu}
\address{(Yu) Institute of Mathematics,
  Academia Sinica, Astronomy-Mathematics
  Building, No. 1, Sec. 4, Roosevelt Road, Taipei 10617, TAIWAN.}
\email{chiafu@math.sinica.edu.tw} 

\address{(Yu) National Center for Theoretical Sciences, 
  Astronomy-Mathematics
  Building, No. 1, Sec. 4, Roosevelt Road, Taipei 10617, TAIWAN.}
\email{chiafu@math.sinica.edu.tw} 

\begin{document}

\date{\today}  \keywords{}

\begin{abstract}

  Extending the results of \cite[Asian J. Math.]{xue-yang-yu:ECNF},
  in \cite[Doc. Math. \textbf{21},
  2016]{xue-yang-yu:sp_as} we calculated  explicitly the number of
  isomorphism classes of superspecial abelian surfaces over an 
  arbitrary finite field of \textit{odd} degree over the prime field
  $\Fp$. A key step 
  was to reduce the
  calculation to the prime field case, and we calculated the number of
  isomorphism classes in each isogeny class through a concrete lattice
  description. In the present paper we treat the \textit{even} 
  degree case by a different method. We first translate the problem by
  Galois cohomology into a
  seemingly unrelated problem of computing 
  conjugacy classes of elements of finite order in arithmetic
  subgroups, which is of independent interest. We then explain how to
  calculate the number of these classes for the arithmetic
  subgroups concerned, 
  and complete the computation in the case of rank two. 
  This complements our earlier results and completes the explicit 
  calculation of superspecial abelian surfaces over finite fields. 


\end{abstract}

\maketitle



\section{Introduction}
\label{sec:intro}


Throughout this paper, $p$ denotes a prime number and $q$ is a power
of $p$. An abelian variety over a field $k$ of characteristic $p$ 
is said to be {\it supersingular} if it is
isogenous to a product of supersingular elliptic curves over an
algebraic closure 
$\bar k$ of $k$; it is said to be {\it superspecial} if it is
isomorphic to a product of supersingular elliptic curves over $\bar
k$. As any supersingular abelian variety is isogenous to a
superspecial abelian variety\footnote{This is well known when the
  ground field $k$ is algebraically closed, and is proved in
  \cite{yu:note_ss} for arbitrary $k$}, it is common to study
supersingular abelian varieties through investigating 
the superspecial abelian varieties.  

Our goal is to calculate explicitly the number of superspecial abelian
surfaces over an arbitrary finite field. This is  motivated by the
search for 
natural 
generalizations of known explicit results of elliptic curves over 
finite fields to abelian surfaces, especially from supersingular
elliptic curves to supersingular abelian surfaces. Thus, studying
superspecial abelian surfaces becomes a vital step for this purpose.
Explicit calculations for supersingular abelian surfaces are certainly
more complicated. 
However, 
if all supersingular cases are
understood, then we would have very good understanding of abelian
surfaces over finite fields, 
because the cases of ordinary and almost ordinary (simple) abelian
surfaces are simpler and have been studied by 
Waterhouse \cite{waterhouse:thesis}. 
That would improve our knowledge from $d=1$
(elliptic curves by Deuring in the 1940's) to 
$d=2$ (abelian surfaces).

In \cite{xue-yang-yu:sp_as} we
calculated explicitly the number of superspecial abelian surfaces over
an arbitrary finite field $\Fq$ of {\it odd} degree over $\Fp$. This extended
our earlier works \cite{xue-yang-yu:ECNF, xue-yang-yu:num_inv} and  
\cite{yu:sp-prime}, which contributed to 
the study of superspecial abelian varieties over finite fields. 
In this paper we treat the even degree case.  
Thus, this complements the
results in loc.~cit.
~and complete an explicit calculation  
of superspecial abelian surfaces over an arbitrary finite field. 
 
A key step in \cite{xue-yang-yu:sp_as} is the reduction 
to the case where the ground field is a prime finite field. This step
is achieved by a Galois cohomology argument. Then we calculate
case-by-case the number of superspecial abelian surfaces in each
isogeny class over $\Fp$. This 
approach works fine when the field $\Fq$ is of odd degree over $\F_p$
because we have a natural concrete lattice description for abelian varieties
over $\Fp$ (see \cite[Theorem 3.1]{yu:sp-prime}).
When the degree $[\Fq:\Fp]$ is even, the Galois cohomology argument 
unfortunately yields no immediate simplification. 
However, it leads to 
a seemingly unrelated problem, which is important but also equally
challenging,  
on counting conjugacy classes of elements of finite order in
arithmetic subgroups. Although the 
connection itself is straightforward, 
it is applicable to a rather general setting;
see Proposition~\ref{1.1}.

For any group $G$, we denote by $\Cl (G)$ the set of conjugacy classes
of 
$G$ and $\Cl_0 (G)\subset \Cl (G)$ the subset of classes of group elements 
of finite order. Let $D=D_{p,\infty}$ be the definite quaternion
$\Q$-algebra 
ramified exactly at $p$ and $\infty$, and $\calO$ be a maximal order in
$D$. 

\begin{prop}\label{1.1}
  Let $\Fq$ be a finite field containing $\bbF_{p^2}$, and $d>1$ be an 
  integer. 
  Then the set of $\Fq$-isomorphism classes of
  $d$-dimensional superspecial abelian varieties over $\Fq$ is in
  bijection with the set $\Cl_0 (\GL_d(\calO))$. 
\end{prop}

By a classical result of Eichler \cite{Eichler-class-no1-1938}, if
$d>1$, then the class number of $\Mat_d(\calO)$ is equal to one, where
$\Mat_d(\calO)$ denotes the ring of $d\times d$ matrices over $\calO$.
Thus, for  $d\ge 2$, any maximal arithmetic subgroup in $\GL_d(D)$ is
conjugate to  
$\GL_d(\calO)$ by an element in $\GL_d(D)$, and hence
Proposition~\ref{1.1} does not depend on the choice of the maximal
order $\calO$.  
The main result of this paper may be stated roughly as follows.


\begin{thm}\label{1.2}
  We have an explicit formula for the cardinality
  of $\Cl_0 (\GL_2(\calO))$. 
\end{thm}

The explicit formula will be given as the sum of certain 
refined terms (according to the classification of isogeny classes) and
an explicit formula for each term is
given in Theorem~\ref{card.1}. Combining with Proposition~\ref{1.1},
we obtain an explicit formula 
for the number of superspecial abelian surfaces over $\Fq\supset
\F_{p^2}$. 
In fact, the method we use for computing $\Cl_0 (\GL_2(\calO))$ 
also provides a way of fining structures of $\Cl_0 (\GL_d(\calO))$ 
for higher $d$; Section~\ref{sec:3.1}. 
Clearly, it would become a hopeless task to carry over all 
the terms as in Theorem~\ref{card.1} for $d=2$ when $d$ increases. 
However, it is still very interesting to 
figure out what are the main term and error terms of 
$\# \Cl_0 (\GL_d(\calO))$ depending on either $d$ or $p$ varies. 
It also makes sense to ask the similar question 
for more general arithmetic subgroups due to a finiteness result of 
Borel; see Theorem~\ref{borel}.      


The main part of this paper is the computation of several  
class numbers of certain (not necessarily maximal) orders in 
some subalgebras of $\Mat_2(D)$ that are the centralizers of 
elements of finite order. The expression of these terms looks 
familiar with the geometric side of a trace formula; however, 
it is not clear
whether the expression is indeed the trace of a  
certain Hecke operator or not. 

This paper is organized as follows. In Section 2, we 
give a proof of Proposition~\ref{1.1}. 
Section 3 describes our main results 
in details. The remaining part of this paper fills in the details of
the computation in Theorem~\ref{1.2}.

\section{Conjugacy classes of elements of finite order}
\label{sec:conj-class-elem}

As in the introduction, for any group $G$, 
we denote by $\Cl (G)$ the set of conjugacy classes of
$G$, and by $\Cl_0 (G) \subset \Cl (G)$ the subset of classes of elements of
finite order in $G$. In this section we reduce the computation of the number of 
$d$-dimensional superspecial abelian varieties 
to that of $\Cl_0 (\GL_d(\calO))$ (Proposition~\ref{1.1}).
Then we study the finiteness of $\Cl_0 (G)$ for a certain special type of 
groups $G$. The latter part is of independent interest,
and will not be used in the rest of this paper.


\subsection{Galois cohomology and forms}
\label{sec:galo-cohom-forms}

Let $X_0$ be a quasi-projective algebraic variety over an arbitrary
field $k$, 
and denote by $\Gamma_k=\Gal(k_s/k)$ the Galois group of $k_s/k$,
where $k_s$ is a separable closure of $k$. 
Let $\Sigma(X_0,k_s/k)$ denote the set of isomorphism classes of 
$k_s/k$-forms of
$X_0$. In other words, $\Sigma(X_0,k_s/k)$ classifies algebraic
varieties $X$ over $k$ 
such that there is an isomorphism $X\otimes_k k_s \simeq X_0\otimes_k
k_s$. 
A well-known result due to Weil states that there is a natural bijection
$\Sigma(X_0,k_s/k)\isoto H^1(\Gamma_k, G)$ of pointed sets, where
$G=\Aut_{k_s}(X_0):=\Aut_{k_s}(X_0\otimes_k k_s)$ is the group of automorphisms of
$X_0\otimes_k k_s$ over $k_s$, 
equipped with the discrete topology and a continuous $\Gamma_k$-action. 
If $\Gamma_k$ acts trivially on $\Aut_{k_s}(X_0)$, namely the
natural inclusion 
$\Aut(X_0)\embed \Aut_{k_s}(X_0)$ is bijective, then
$H^1(\Gamma_k, G)=\Hom(\Gamma_k, G)/G$, where $G$ acts on the set $\Hom(\Gamma_k, G)$ of continuous homomorphisms by conjugation.
In addition, if $\Gamma_k$ is isomorphic to the profinite group 
$\wh \Z=\varprojlim \Z/m\Z$, one obtains 
a natural bijection of pointed sets:
\begin{equation}
  \label{eq:conj.1}
  \Sigma(X_0,k_s/k)\isoto \Cl_0(G), \quad G=\Aut(X_0). 
\end{equation}

Let $X$ be an abelian variety over $k$ and $k'/k$ be a field extension. 
Denote by $\End_{k'}(X)=\End_{k'}(X\otimes_k k')$ 
the endomorphism ring  of $X\otimes_k k'$ 
over $k'$; we also write $\End(X)$ for $\End_k(X)$. 
The endomorphism algebra of $X\otimes_k k'$ is defined to be $\End(X\otimes_k k')\otimes_\Z \Q$ and denoted by
$\End^0_{k'}(X)$. 
Applying Weil's result to abelian varieties over finite fields, 
one obtains the following proposition.

\begin{prop}\label{conj.2}
  Let $X_0$ be an abelian variety over a finite field $\Fq$ such that
  the endomorphism algebra $\End^0_{\Fqbar}(X_0)$ is equal to
  $\End^0(X_0)$, and let $G=\Aut(X_0)$. Then there is a natural
  bijection of pointed sets $\Sigma(X_0,\Fqbar/\Fq)\simeq \Cl_0(G)$.  
\end{prop}

Note that the group $G$ in Proposition~\ref{conj.2} is an arithmetic
subgroup of 
the reductive group $\ul G$ over $\Q$ such that $
\ul G(R)=(\End^0(X_0)\otimes_\Q R)^\times$ for any $\Q$-algebra $R$. \\

\begin{proof}[\bf Proof of Proposition~\ref{1.1}]
We choose a supersingular
elliptic curve $E_0$ over $\F_{p^2}$ with $\End(E_0)= \calO$ under
an isomorphism $\End^0(E_0)\simeq D$. Put $X_0=E_0^d \otimes_{\F_{p^2}}
  \Fq$, then we have $G=\Aut(X_0)=\GL_d(\calO)$ and the Galois group
  $\Gamma_{\Fq}$ acts trivially on $G$. By Proposition~\ref{conj.2},
  there is a natural bijection $\Sigma(X_0,\Fqbar/\Fq)\isoto \Cl_0
  (G)$. As $X_0$ is 
  superspecial of dimension $d>1$, for any $d$-dimensional superspecial
  abelian variety $X$ over $\Fq$ there is an isomorphism
  $X\otimes_{\Fq} \Fqbar \simeq X_0\otimes_{\Fq} \Fqbar$;
  see~\cite[Section 1.6, p.13]{li-oort}. Thus, $\Sigma(X_0,k_s/k)$
  classifies the $d$-dimensional superspecial abelian 
  varieties over $\Fq$ up to $\Fq$-isomorphism. This completes the
  proof of the proposition. 
\end{proof}

\subsection{Finiteness of $\Cl_0 (G)$ for some groups $G$}
\label{sec:conj.2}

We take the opportunity to discuss the finiteness of $\Cl_0 (G)$
for a group $G$ that is of the form $H(F)$ for a reductive group $H$
over a global or
local field $F$, or an arithmetic subgroup of $H(F)$ 
for a global field $F$. 
The following is a fundamental result due to Borel; see
\cite[Section 5]{Borel:1962}. 

\begin{thm}\label{borel}
Let $G$ be a reductive group over a number field $F$, and
$\Gamma\subset G(F)$ an $S$-arithmetic subgroup, where $S$ is a finite
set of places of $F$ containing all the archimedean ones.  
Then there are only finitely many finite subgroups of $\Gamma$ up to 
conjugation by $\Gamma$. In particular, $\Cl_0 (\Gamma)$ is finite. 
\end{thm}

We could not find the reference for 
a similar result of Theorem~\ref{borel} under the condition that $F$
is a global function field. 



\begin{prop} \ 

  {\rm (1)} Suppose $G$ is a linear algebraic group over a 
  non-archimedean local field $F$ of characteristic zero. 
  Then $\Cl_0 (G(F))$ is finite. 

  {\rm (2)} Let $G$ be a linear algebraic  group over $\R$. 
     Then the set $\Cl_0 (G(\R))$ is
      infinite if and only if $G(\R)$ contains
      a non-trivial compact torus $S$.

\end{prop}
\begin{proof}
  (1) Since ${\rm char}\, F=0$, every element in $G(F)$ of 
  finite order is semisimple and hence     
  it is contained in a maximal $F$-torus $T$ of $G$. By \cite[Section
  6.4 Cor.~3, p.~320]{platonov-rapinchuk}, there are
  only finitely many maximal $F$-tori up to conjugation by $G(F)$. 
  Therefore, one
  reduces the statement to the case where $G=T$ is a torus. 
  Choose a finite
  extension $K$ of $F$ over which $T$ splits. Then one has $T(F)\subset
  (K^\times)^d$, where $d=\dim T$. Since there are only finitely many
  roots of unity in $K^\times$, the subgroup $T(F)_{\rm tors}=\Cl_0(
  T(F))$ is finite.

 (2) Suppose that every $\R$-torus $T$ of $G$ is split. 
 Then we use the argument in
 (1) to prove that $\Cl_0 (G(\R))$ is finite, because for a split
 torus $T$, the set $\Cl_0 (T(\R))=T(\R)_{\rm tors}$ is finite.  
Now we prove the other direction. 
Suppose that $G(\R)$ contains a non-trivial compact torus $S$. 
Then for any positive integer $n$ 
there is an element of order $n$ in $S$. 
Since elements of different orders are not conjugate, the set $\Cl_0 (G(\R))$ is infinite. This proves 
the proposition.
\end{proof}

\begin{prop}
Let $A$ be a finite-dimensional semisimple algebra over a number field
$F$. Then $\Cl_0 (A^\times)$ is finite. 
\end{prop}
\begin{proof}
For each positive integer $n$, denote by $\Hom_{F}(F[t]/(t^n-1),A)$
the set of $F$-algebra homomorphisms from $F[t]/(t^n-1)$ to $A$, and
by $\Hom^*_{F}(F[t]/(t^n-1),A)$ the subset consisting of all maps
$\varphi$ 
with $\ord (\varphi(t))=n$. The group $A^\times$ acts on
$\Hom_{F}(F[t]/(t^n-1),A)$ by conjugation, and we have orbit spaces  
$$\Hom^*_{F}(F[t]/(t^n-1),A)/A^\times\subset
\Hom_{F}(F[t]/(t^n-1),A)/A^\times.$$  
Let $\Cl_0(n,A^\times)$ denote the set of conjugacy classes of
elements of order $n$ in $A^\times$. Clearly this set agrees with the
set $\Hom^*_{F}(F[t]/(t^n-1),A)/A^\times$. 

Since $A$ is separable over $F$ and $F[t]/(t^n-1)$ is semisimple, 
an extension
of the Noether-Skolem theorem (\cite[Theorem 2]{pop-pop-1985}, also see
\cite[Theorem 1.4]{yu:embed}) states 
the set $\Hom_{F}(F[t]/(t^n-1),A)/A^\times$ is finite. Thus,  
$\Cl_0(n,A^\times)$ is finite for each $n$. Note that 
$\Cl_0(A^\times )$ is the union of $\Cl_0(n,A^\times)$ for all $n$, 
and that  $\Cl_0(n,A^\times)$ is empty for almost all $n$, because the
degree of $\Q(\zeta_n)$ is unbounded when $n$ goes large, 
This proves the finiteness of  $\Cl_0(A^\times )$.
\end{proof}

Now we provide an example showing that $\Cl_0 (G(F))$ can be infinite
for a connected reductive group $G$ over a number field $F$. Take
$G=\SL_2$ and $F=\Q$. Consider the subset $\Cl_0(4, \SL_2(\Q))\subset
\Cl_0(\SL_2(\Q))$ of classes of order $4$. We choose a base point
$\xi_0=\begin{pmatrix} 0 & -1 \\ 1 & 0 \end{pmatrix}$ and set
$K:=\Q(\xi_0)$, which is isomorphic to $\Q(\sqrt{-1})$. Every
element $\xi\in \SL_2(\Q)$ of order $4$ is conjugate to $\xi_0$ by an
element $g_1$ in $\GL_2(\Q)$, i.e. $\xi=g_1 \xi_0 g_1^{-1}$. Two
elements $g_1$ and $g_2$ in $\GL_2(\Q)$ give rise to the same element
$\xi$ if and only if $g_2=g_1 z$ for some element $z\in
K^\times$. Moreover, suppose $\xi_1$ and $\xi_2$ are two elements in
$\SL_2(\Q)$ of order $4$ presented by $g_1$ and $g_2$,
respectively. Then $\xi_1$ and $\xi_2$ are conjugate in $\SL_2(\Q)$ if
and only if $g_2=h g_1 z$ for some elements $h\in \SL_2(\Q)$ and $z\in
K^\times$. Therefore, we have proved a bijection 
\begin{equation}\label{eq:SL2}
\Cl_0(4,\SL_2(\Q))\simeq \SL_2(\Q)\backslash \GL_2(\Q)/K^\times. 
\end{equation}

Taking the determinant, we have $\Cl_0(4,\SL_2(\Q))\simeq \Q^\times/
\Nm_{K/\Q}(K^\times)$. Note that $\Nm_{K/\Q}(K^\times)$ consists of
all non-zero elements of the form $a^2+b^2$ with $a, b\in \Q$. 
By basic number theory, we obtain the following result. 

\begin{prop}
The set $\Cl_0(4,\SL_2(\Q))$ is in bijection with the $\F_2$-vector
space generated by $-1$ and prime elements $p$ with $p\equiv 3 \pmod
4$. In particular, the set $\Cl_0(4,\SL_2(\Q))$ is infinite.  
\end{prop}

\begin{rem}
Another way to interpret the previous example is through the point of
view of \textit{stable conjugacy classes}.  Let $G$ be a connected
reductive group over $F$ as before. Two elements $\xi_1, \xi_2\in
G(F)$ are said to be \emph{stably conjugate} if there exists $g\in
G(\overline{F})$ such that $\xi_1=g\xi g^{-1}$.  Let $G_\xi$ be the
centralizer of $\xi \in G(F)$.  Langlands
\cite{Langlands-stable-conj-1979} establishes a bijection between the
set of conjugacy classes within the stable conjugacy class of $\xi$
and $\ker(H^1(F, G_\xi)\to H^1(F, G))$.  In the example where
$G=\SL_2$ and $F=\Q$, every element of order $4$ in $\SL_2(\Q)$ is
stably conjugate to $\xi_0$.  Since $H^1(\Q, \SL_2)=\{1\}$ and
$G_{\xi_0}$ coincides with the norm 1 torus
$T:=\ker\big(\Res_{K/\Q}(\G_{m, K})\xrightarrow{\Nm_{K/\Q}}
\G_{m,\Q}\big)$, we recover the result  
\[\Cl_0(4, \SL_2(\Q))\simeq H^1(\Q,
T)=\Q^\times/\Nm_{K/\Q}(K^\times). \] 
\end{rem}


We mention Springer and Steinberg \cite{springer-steinberg-1970} and
Humphreys \cite{Humphreys:1995} as 
important references for  
conjugacy classes of linear algebraic groups.

\section{The Cardinality of  $\Cl_0 (\GL_2(\calO))$}
\label{sec:stat-main-results}

Let $D$ be a finite-dimensional central division $\Q$-algebra, and
$\calO$ a maximal order in $D$. Fix an integer $d>1$.  We explain the
strategy for calculating the cardinality of $\Cl_0 (\GL_d(\calO))$,
based on the lattice description of conjugacy classes in
 \cite[Section~6.4]{xue-yang-yu:sp_as}.  
As remarked right after Proposition~\ref{1.1}, $\abs{\Cl_0 (\GL_d(\calO))}$ depends only on $d$ and $D$, not on the
choice of the maximal order $\calO$. So it makes sense to set
$H(d, D):=\abs{\Cl_0 (\GL_d(\calO))}$.  The strategy is carried out in
detail for the case $d=2$ and $D=D_{p,\infty}$ in subsequent sections
under a mild condition on $p$ (see Remark~\ref{rem:assumption}), and the resulting formula for
$H(2, D_{p,\infty})$ is stated in Theorem~\ref{card.1}. As a convention, $\bbN$ denotes the set of strictly positive integers, and $\Z_{\geq 0}$ the set of nonnegative ones. 

\subsection{The general strategy}\label{sec:3.1} 
Given an element $x\in \GL_d(\calO)$
of finite order, its minimal polynomial over $\Q$ is of the form
\begin{equation}
  P_\uln(T)=\Phi_{n_1}(T)  \cdots \Phi_{n_r}(T), \quad
  1\le n_1<\cdots <n_r
\end{equation}
for some $r$-tuple $\uln=(n_1, \ldots, n_r)\in \bbN^r$, where
$\Phi_n(T)\in \Z[T]$ denotes the $n$-th cyclotomic polynomial.  For
simplicity, we denote the set of strictly increasing $r$-tuples of positive
integers by $\brN^r$.  Let
$C(\uln)\subseteq \Cl_0(\GL_d(\calO))$ be the subset of conjugacy
classes with minimal polynomial $P_\uln(T)$.  The subring
$\Z[x]\subset \Mat_d(\calO)$ (resp.~subalgebra
$\Q[x]\subset \Mat_d(D)$) generated by $x$ is isomorphic to $A_\uln$
(resp.~$K_\uln$) defined as follows
\begin{equation}
  \label{eq:3}
 A_\uln:=\frac{\Z[T]}{\prod_{i=1}^r \Phi_{n_i}(T)}, \qquad
K_\uln:=\frac{\Q[T]}{\prod_{i=1}^r \Phi_{n_i}(T)}\cong \prod_{i=1}^r \Q[T]/(\Phi_{n_i}(T)). 
\end{equation}
If $r=1$, we omit the underline in $\uln$ and write $A_n$ and $K_n$
instead. Hence $K_\uln\cong \prod_{i=1}^r K_{n_i}$, but this decomposition
does \textit{not} hold for $A_\uln$ in general. Let $\calO^\opp$
(resp.~$D^\opp$) be
the opposite ring of $\calO$ (resp. $D$). We define
\begin{equation}
  \label{eq:12}
\scrA_\uln:=A_\uln\otimes_\Z\calO^\opp, \quad \text{and} \quad
\scrK_\uln:=K_\uln\otimes_\Q D^\opp.
\end{equation}
Clearly, $\scrA_\uln$ is an order in the semisimple $\Q$-algebra
$\scrK_\uln\cong \prod_{i=1}^r\scrK_{n_i}$.  Each $\scrK_n$ is a
central simple $K_n$-algebra, whose left simple module is denoted by
$W_n$. The dimension $e(n)$ of $W_n$ as a left $D^\opp$-space (or
equivalently, a right $D$-space) is also the smallest $e\in \bbN$ such
that there exists an embedding $K_n\hookrightarrow \Mat_e(D)$.


Let $V=D^d$ be the right $D$-space of column vectors, and
$M_0=\calO^d\subset V$ the standard $\calO$-lattice in $V$.  Then
$\End_\calO(M_0)=\Mat_d(\calO)$, acting on $M_0$ from the left by
multiplication.  The conjugacy class $[x]\in C(\uln)$ equips
$M_0$ with a faithful $(A_\uln, \calO)$-bimodule structure, or
equivalently, a faithful left $\scrA_\uln$-module structure.
Similarly, $V$ is equipped with a faithful left $\scrK_\uln$-module
structure.  The decomposition of $K_\uln$ in (\ref{eq:3}) induces a decomposition
\begin{equation}
  \label{eq:10}
V=\oplus_{i=1}^r V_{n_i}, 
\end{equation}
where each $V_{n_i}$ is a \textit{nonzero} $\scrK_{n_i}$-module. Hence
$V_{n_i}\simeq (W_{n_i})^{m_i}$ for some $m_i\in \bbN$. Comparing the
$D$-dimensions, we get
\begin{equation}
  \label{eq:4}
  d=m_1e(n_1)+\cdots +m_re(n_r). 
\end{equation}
The $r$-tuple $\ulm=(m_1, \ldots, m_r)\in \bbN^r$ will be called the
\textit{type} of the left $\scrK_\uln$-module $V$, and the pair
$(\uln, \ulm)$ the \textit{type} of the conjugacy class
$[x]\in C(\uln)\subseteq \Cl_0(\GL_d(\calO))$.

A pair of $r$-tuples $(\uln, \ulm)\in \brN^r\times \bbN^r$
is said to be $d$-\textit{admissible} if it
satisfies equation (\ref{eq:4}).  Let
$C(\uln, \ulm)\subseteq \Cl_0(\GL_d(\calO))$ be the subset of
conjugacy classes of type $(\uln, \ulm)$. Then we have
\begin{equation}
  \label{eq:5}
  \Cl_0(\GL_d(\calO))=\coprod_\uln C(\uln)=\coprod_{(\uln,
    \ulm)}C(\uln, \ulm), 
\end{equation}
where $(\uln, \ulm)$ runs over all $d$-admissible pairs.  The same
proof of \cite[Theorem~6.11]{xue-yang-yu:sp_as} establishes a
bijection between $C(\uln, \ulm)$ and the set $\scrL(\uln, \ulm)$ of
isomorphism classes of $\scrA_\uln$-lattices in the left
$\scrK_\uln$-module $V$ of type $\ulm$. The latter set is finite
according to the Jordan-Zassenhaus Theorem \cite[Theorem 24.1,
p.~534]{curtis-reiner:1}.
Put $o(\uln):=\abs{C(\uln)}$
and $o(\uln, \ulm):=\abs{C(\uln, \ulm)}=\abs{\scrL(\uln, \ulm)}$. It follows from (\ref{eq:5})
that
\begin{equation}
  \label{eq:9}
  H(d, D)=\abs{\Cl_0(\GL_d(\calO))}=\sum_{\uln} o(\uln)=\sum_{(\uln,
    \ulm)}o(\uln, \ulm). 
\end{equation}

Now fix a $d$-admissible pair $(\uln, \ulm)\in \brN^r\times \bbN^r$
and a left $\scrK_\uln$-module $V$ of type $\ulm$. The isomorphism
class of an $\scrA_\uln$-lattice $\Lambda\subset V$ is denoted by
$[\Lambda]$.  Two $\scrA_\uln$-lattices $\Lambda_1, \Lambda_2\subset V$
are isomorphic if and only if there exists
$g\in \End_{\scrK_\uln}(V)^\times$ such that $\Lambda_1=\Lambda_2g$
(In particular,  $\End_{\scrK_\uln}(V)$ acts on $V$ from the
right). Clearly, 
\begin{equation}
  \label{eq:11}
\End_{\scrK_\uln}(V)=\oplus_{i=1}^r \End_{\scrK_{n_i}}(V_{n_i}),
\ \text{and} \ \End_{\scrK_{n_i}}(V_{n_i})\sim
\scrK_{n_i}=K_{n_i}\otimes_\Q D^\opp,
\end{equation}
where $\sim$ denotes the Brauer equivalence of central simple
algebras. On the other hand,
$\End_{\scrK_{n_i}}(V_{n_i})^\opp$ is canonically isomorphic to the
centralizer of $K_{n_i}$ in $\End_D(V_{n_i})$. So by the Centralizer
Theorem \cite[Theorem 3.15]{Farb-Dennis-1993},
\begin{gather}
  [K_{n_i}:\Q]^2[\End_{\scrK_{n_i}}(V_{n_i}):K_{n_i}]=[\End_D(V_{n_i}):\Q]=[D:\Q](m_ie(n_i))^2.   \label{eq:13}
\end{gather}
The structure of $\End_{\scrK_\uln}(V)$ is completely determined  by
(\ref{eq:11}) and (\ref{eq:13}). 

For each prime $\ell\in \bbN$, let $\Lambda_\ell$ be the $\ell$-adic
completion of $\Lambda$, and $\scrL_\ell(\uln, \ulm)$ the set of
isomorphism classes of $\scrA_{\uln,\ell}$-lattices in $V_\ell$.  The profinite
completion $\Lambda\mapsto \wh \Lambda=\prod_\ell \Lambda_\ell$
induces a surjective map
\begin{equation}
  \label{eq:14}
\Psi: \scrL(\uln, \ulm)\to \prod_\ell \scrL_\ell(\uln, \ulm).  
\end{equation}
 For almost all
primes $\ell$, the order $\scrA_{\uln, \ell}$ is maximal in
$\scrK_{\uln, \ell}$, in which case $\scrL_\ell(\uln, \ulm)$ is a
singleton by \cite[Theorem~26.24]{curtis-reiner:1}. So the right hand
side of (\ref{eq:14}) is essentially a finite product. Two lattices
$\Lambda_1$ and $\Lambda_2$ are said to be in the same \textit{genus}
if $\Psi([\Lambda_1])=\Psi([\Lambda_2])$, that is, if they are locally
isomorphic at every prime $\ell$.  The fibers of $\Psi$ partition
$\scrL(\uln, \ulm)$ into a disjoint union of genera.  More explicitly,
for each element
$\bbL=(\bbL_\ell)_\ell\in \prod_\ell \scrL_\ell(\uln, \ulm)$, let
\begin{equation}
  \label{eq:15}
\scrL(\uln, \ulm, \bbL):=\Psi^{-1}(\bbL)=\{[\Lambda]\in
\scrL(\uln,\ulm) \mid  [\Lambda_\ell]= 
\bbL_\ell, \forall \ell\,\}. 
\end{equation}
Then $\scrL(\uln, \ulm)=\coprod_\bbL \scrL(\uln, \ulm, \bbL)$, where
$\bbL$ runs over elements of $\prod_\ell \scrL_\ell(\uln, \ulm)$. 

Lastly, we pick an $\scrA_\uln$-lattice $\Lambda\subset V$ with
$[\Lambda]\in \scrL(\uln, \ulm, \bbL)$ and write $O_\Lambda$ for its
endomorphism ring
$\End_{\scrA_\uln}(\Lambda)\subset \End_{\scrK_\uln}(V)$.  It follows
from \cite[Proposition~1.4]{Swan-1988} that $\scrL(\uln, \ulm, \bbL)$
is bijective to the set of locally principal right ideal classes of $O_\Lambda$. In
particular,
\begin{equation}
  \label{eq:16}
  \abs{\scrL(\uln, \ulm, \bbL)}=h(O_{\Lambda}). 
\end{equation}
Another choice  $\Lambda'$ with
$[\Lambda']\in \scrL(\uln, \ulm, \bbL)$ produces an endomorphism ring
$O_{\Lambda'}$ locally conjugate to $O_\Lambda$  at every
prime $\ell$, and hence gives rise to the same class number $h(O_{\Lambda'})=h(O_\Lambda)$. If
$\scrA_\uln$ is maximal at $\ell$, then $(O_\Lambda)_\ell$ is a
maximal order in $\End_{\scrK_\uln}(V)_\ell=\End_{\scrK_{\uln,\ell}}(V_\ell)$

In summary, the calculation of $H(d,D)$ is separated into $3$ steps: 
\begin{enumerate}
\item For each $1\leq r\leq d$, list the set $\calT(d,r)$ of all
  $d$-admissible pairs $(\uln, \ulm)\in \brN^r\times \bbN^r$. We set
  $\calT(r)=\calT(d,r)$ if $d$ is clear from the context.
\item For each $(\uln, \ulm)$, classify the genera of
  $\scrA_{\uln}$-lattices in the left $\scrK_\uln$-module $V$ of type
  $\ulm$. This amounts to classifying the isomorphism classes of
  $\scrA_{\uln, \ell}$-lattices in $V_\ell$.  Only the primes $\ell$
  with $\scrA_{\uln, \ell}$  non-maximal come in to play.
\item For each genus, write down (at least locally) the endomorphism ring of an
  lattice member and calculate its class number. The sum
  of all these class numbers is $H(d,D)$.
\end{enumerate}

 \begin{rem}\label{rem:simplifications}
We make a couple of simplifications for the calculations.

\noindent(i) The center $Z(\GL_d(\calO))=\{\pm 1\}$ acts on
$\Cl_0(\GL_d(\calO))$ by multiplication, and induces a bijection
between $C(\uln)$ and $C(\uln^\dagger)$, where $\uln^\dagger$ is obtained by first
defining an intermediate $r$-tuple $\uln^\ddagger:=(n_1^\ddagger, \ldots, n_r^\ddagger)$
with
\begin{equation}
  \label{eq:17}
 n_i^\ddagger=
\begin{cases}
  2n_i \qquad &\text{if } 2\nmid n_i,\\
  n_i  \qquad &\text{if } 4\mid n_i,\\
  n_i/2 \qquad &\text{otherwise},
\end{cases}
\end{equation}
for each $1\leq i\leq r$, and then rearranging its entries in ascending order. For example, if
$\uln=(3,4)$, then $\uln^\dagger=(4, 6)$. Thus $o(\uln)=o(\uln^\dagger)$ and only
one of them needs to be calculated.

\noindent(ii) Let $u$ be the reduced degree of $D$ over $\Q$, and
$\Lambda$ an $\scrA_\uln$-lattice in the $\scrK_\uln$-module $V$ of
type $\ulm$.  For almost all primes $\ell$, we have
$\calO^\opp\otimes \Z_\ell\simeq \Mat_u(\Z_\ell)$, and hence
$\scrA_{\uln,\ell}\simeq \Mat_u(A_{\uln,\ell})$. Fix such an $\ell$. It then follows from Morita equivalence that
$\Lambda_\ell\simeq (\Lambda'_\ell)^u$ and $V_\ell\simeq (V'_\ell)^u$,
where $\Lambda'_\ell$ is an $A_{\uln, \ell}$-lattice in the
$K_{\uln, \ell}$-module $V'_\ell=\prod_{i=1}^r V'_{n_i, \ell}$. Each
$V'_{n_i, \ell}$ is a free $K_{n_i,\ell}$-module of rank
\begin{equation}
  \label{eq:18}
\dim_\Q(D^{m_ie(n_i)})/(u[K_{n_i}:\Q])=um_ie(n_i)/\varphi(n_i).
\end{equation}
The association $\Lambda_\ell\mapsto \Lambda_\ell'$ establishes a
one-to-one correspondence between $\scrL_\ell(\uln, \ulm)$ and the set
of isomorphism classes of $A_{\uln, \ell}$-lattice in
$V'_\ell$. Moreover,
$\End_{\scrA_{\uln, \ell}}(\Lambda_\ell)\cong \End_{A_{\uln,
    \ell}}(\Lambda'_\ell)$.
This reduces the classification of lattices over the
\emph{non-commutative} order $\scrA_{\uln, \ell}$ to that over the
\emph{commutative} order $A_{\uln, \ell}$, which is much easier.  We
make use of this simplification in Section~\ref{sec:non-elem-case}; see Table~\ref{tab:V-ell}. 
 \end{rem}

 \begin{rem}\label{rem:geom-meaning}
   When $D=D_{p, \infty}$, many numerical invariants discussed in this
   subsection admit natural geometric meanings. Assume that $\F_q$ has
   even degree $a:=[\F_q:\F_p]$ over its prime field as in
   Proposition~\ref{1.1}.  Recall that an algebraic integer $\pi\in \overline{Q}$ is
   said to be a \emph{Weil $q$-number} if $\abs{\iota(\pi)}=q^{1/2}$
   for every embedding $\iota: \Q(\pi)\hookrightarrow \C$. One class
   of examples is given by $\pi_n:=(-p)^{a/2}\zeta_n$ for
   $n\in \bbN$, where $\zeta_n$ denotes a primitive $n$-th root of
   unity. By the Honda-Tate theorem, there is a unique simple abelian
   variety $X_n$ over $\F_q$ up to isogeny corresponding to the
   $\Gal(\overline{\Q}/\Q)$-conjugacy class of $\pi_n$. It is known
   \cite[Subsections~6.2]{xue-yang-yu:sp_as} that $X_n$ is a
   supersingular abelian variety of dimension $e(n)$. Fix an integer
   $d>1$ and let $(\ul n, \ul m)\in \brN^r\times \bbN^r$ be a
   $d$-admissible pair.  By \cite[Theorem~6.11]{xue-yang-yu:sp_as},
   $o(\ul n, \ul m)$ counts the number of $\F_q$-isomorphism classes
   of $d$-dimensional \emph{superspecial} abelian varieties over
   $\F_q$ in the isogeny class of
   $X_{\uln, \ulm}:=\prod_{i=1}^r (X_{n_i})^{m_i}$. In particular, $r$
   measures the number of $\F_q$-isotypic parts of $X_{\uln, \ulm}$.
   By \cite[Lemma~6.3]{xue-yang-yu:sp_as}, we have
   $\End^0(X_{\uln, \ulm})\simeq \End_{\scrK_\uln}(V)^\opp$, where $V$ is a
   left $\scrK_\uln$-module of type $\ulm$.  
 \end{rem}

\subsection{Explicit formulas for $H(2, D_{p, \infty})$}
\label{sec:card.2}
First, we list all $2$-admissible pairs
$(\uln, \ulm)\in \brN^r\times \bbN^r$ for $r=1,2$. Note that
$e(n)\leq 2$ only if $\varphi(n)=[K_n:\Q]\leq 4$, 
i.e. $n\in\{1,2,3,4,5,6,8,10,12\}$.  More explicitly, 
\begin{itemize}
\item if $n\in \{1,2\}$, then $K_n=\Q$ and $e(n)=1$;
\item if $n\in \{3, 4, 6\}$, then $[K_n:\Q]=2$. We have $e(n)=2$ if
  $p$ splits in $K_n$, and $e(n)=1$ otherwise. 
\item if $n\in \{5, 8, 10, 12\}$, then $[K_n:\Q]=4$ and $e(n)\geq
  2$. The equality holds if and only if $p$ does \textit{not} split
  completely in $K_n$. 
\end{itemize}

Thus we have 
\begin{gather*}
  \calT(1)=\{(n,m)\in \bbN\times \bbN\mid n\in\{1,2,3,4,5,6,8,10,12\}, me(n)=2\},\\
  \calT(2)=\{((n_1, n_2),(m_1, m_2))\in \brN^2\times \bbN^2\mid
 n_1<n_2,  n_i\in\{1,2,3,4,6\}, m_ie(n_i)=1\}.
\end{gather*}
For each $\uln\in \brN^r$ with $r=1, 2$, there is at most one
$\ulm\in \bbN^r$ such that $(\uln, \ulm)$ is $2$-admissible. So
we omit $\ulm$ from the notation $\scrL(\uln, \ulm)$ and write
$\scrL(\uln)$ instead. Similarly, we discuss only the value of
$o(\uln)$ rather than that of $o(\uln, \ulm)$. Thus there should be no ambiguity of the notation $o(1,2)$ for $\uln=(1,2)$.
By Remark~\ref{rem:simplifications}(i), we have
\begin{gather}
  o(1)=o(2)=1,\ o(3)=o(6), \ o(5)=o(10);\label{eq:43}\\
  o(1,3)=o(2,6), \  o(1,4)=o(2,4), \ o(1,6)=o(2,3), \
o(3,4)=o(4,6).\label{eq:45}
\end{gather}

Let $q=p^a$ be an \emph{even} power of $p$. By
Remark~\ref{rem:geom-meaning}, each $o(\uln)$ above
counts the number of isomorphism classes of superspecial abelian
surfaces over $\F_q$ in a supersingular isogeny class
determined by $\uln\in \brN^r$. If $r=1$ so that $\uln=n\in \bbN$,
then the isogeny class is isotypic; it is even simple if $e(n)=2$. If $r=2$, then the
isogeny class is non-isotypic: every member is $\F_q$-isogenous
to a product of two mutually non-isogenous supersingular elliptic
curves over $\F_q$.  



\begin{thm}\label{card.1}
  Let $D=D_{p,\infty}$ be the quaternion $\Q$-algebra ramified exactly
  at $p$ and $\infty$, and $\calO$ a maximal order in $D$. We have
\begin{equation}
  \label{eq:card.8}
  \begin{split}
H(2, D_{p,\infty})=&\abs{\Cl_0(\GL_2(\calO))}  =2+2o(3)+o(4)+2o(5)+o(8)+o(12) \\
          &+o(1,2)+2o(2,3)+2o(2,4)+2o(2,6)+2o(3,4)+o(3,6),
  \end{split}
\end{equation}
where the value of each $o(\uln)$ is as follows:
\begin{itemize}
\itemsep4pt
\item $o(3)=2-\left ( \frac{-3}{p} \right )$;

\item $o(4)=2-\left (\frac{-4}{p} \right )$;

\item $o(5)=
  \begin{cases}
   1 & \text{if } p=5;\\
   0 & \text{if } p\equiv 1 \phantom{,1} \pmod{5}; \\
   2 &  \text{if } p\equiv 2,3 \pmod{5}; \\
   4 &  \text{if } p\equiv 4 \phantom{,1} \pmod{5};
  \end{cases}$

\item  $o(8)=
  \begin{cases}
   1 & \text{if } p=2; \\
   0 & \text{if } p\equiv 1 \phantom{,1,1} \pmod{8}; \\
   4 & \text{if } p\equiv 3,5,7   \pmod{8}; 
  \end{cases}$

\item $o(12)=
  \begin{cases}
     3 & \text{if } p=2,3;\\
     0 & \text{if } p\equiv 1 \phantom{,1,11}  \pmod{12}; \\
     4 & \text{if } p\equiv 5,7,11 \pmod{12}; 
  \end{cases}$ 


\item $o(1,2)=
  \begin{cases}
3 \quad \text{if}\quad p=3;\\[5pt] 
    \begin{aligned}
\frac{(p-1)^2}{9} &+
  \frac{p+15}{18}\left(1-\left ( \frac{-3}{p} \right)\right) +
  \frac{p+2}{6}\left(1-\left ( \frac{-4}{p} \right)\right) \\
  &+  \frac{1}{6}\left(1-\left ( \frac{-3}{p} \right )\right)\left(1-\left
      ( \frac{-4}{p} \right )\right)\quad  \text{if } p\neq 3;\\
    \end{aligned}
  \end{cases}$

\item  $o(2,3) =\left(1-\left(\frac{-3}{p}\right)\right)\left(\frac{p-1}{12}+\frac{1}{3}\left(1-\left(\frac{-3}{p}\right)\right)+\frac{1}{4}\left(1-\left(\frac{-4}{p}\right)\right)\right)$;

\item  $o(2,4) =\left( \frac{p+3}{3}-\frac{1}{3}\left (
      \frac{-3}{p} \right)\right)\left(1-\left ( \frac{-4}{p}
    \right)\right)$;

\item $o(2,6)=\left(\frac{5p+18}{12}+\frac{1}{3}\left (
       \frac{-3}{p}\right)-\frac{1}{4}\left(\frac{-4}{p}\right)\right)\left(1-\left (
       \frac{-3}{p} \right)\right)$;


\item  $o(3,4) = \left(1-\left ( \frac{-3}{p}
    \right)\right)\left(1-\left ( \frac{-4}{p} \right)\right)$;

\item $o(3,6) = 2\left(1-\left ( \frac{-3}{p}
    \right)\right)^2$.
\end{itemize} 
\end{thm}

\begin{cor}
Keeping the notation of Theorem~\ref{card.1}, we have 
  \begin{equation}
    \label{eq:cor}
    \lim_{p\to \infty} \frac{H(2,D_{p,\infty})}{p^2/9}=1.
  \end{equation}
\end{cor}
\begin{proof}
By Theorem~\ref{card.1},  the dominant term of $H(2,D_{p,\infty})$ is $o(1,2)$, which
  is asymptotic to $(p-1)^2/9$ as $p$ tends to infinity.
\end{proof}









\begin{rem}
  Karemaker and Pries \cite[Proposition 7.2]{karemaker-pries-2017}
  give a full classification of the types of principally polarized
  simple supersingular abelian surfaces $(A,\lambda)$ over a finite
  field $\Fq$ with $\Aut_{\Fqbar}(A,\lambda)=\Z/2\Z$. They also prove
  \cite[Proposition 7.6]{karemaker-pries-2017} that if $p\ge 3$, then
  the proportion of $\F_{p^t}$-rational points of the supersingular locus
  $\calA_{2,\mathrm{ss}}$ which represent $(A,\lambda)$ with
  $\Aut_{\Fpbar}(A,\lambda)\neq \Z/2\Z$ tends to zero as
  $t\to \infty$.  They ask whether or not the majority of principally
  polarized supersingular abelian surfaces over $\F_{p^t}$ are those
  with normalized Weil numbers $(1,1,-1,-1)$. From Theorem~\ref{card.1}
  and \cite[Theorems 1.1 and 1.2]{xue-yang-yu:sp_as} we see that the
  proportion of {\it superspecial} abelian surfaces over $\F_{p^t}$ with
  normalized Weil numbers $(1,1,-1,-1)$ (with $t$ fixed) tends to one
  as $p\to \infty$. However, to deduce a similar result for
  supersingular abelian surfaces, one could use the argument of
  \cite{xue-yu:counting} where we compute the size of the
  isogeny class corresponding to the Weil number $\sqrt{p^t}$ with odd
  $t$.
\end{rem}

In the calculations for Theorem~\ref{card.1}, we make frequent use of
\emph{Eichler orders}, so let us briefly recall the definition and basic
properties.  Let $R$ be a Dedekind domain with fractional field $F$,
and $D$ be a quaternion $F$-algebra (not necessarily division). An
$R$-order $O$ in $D$ is called an \emph{Eichler order} if it can be
written as the intersection of two maximal orders; see
\cite[Corollary~2.2]{MR693798} for an equivalent characterization of
Eichler orders. The $R$-ideal index \cite[Chapter~III, \S1]{Serre_local} of an Eichler order $O$ relative to any maximal $R$-order is called the \emph{level} of $O$. 
Assume further that $R$ is a complete discrete
valuation ring, and let $\pi$ be a local parameter of $R$. If $D$ is
division, then there is a unique maximal order in $D$, hence a unique
Eichler order $O$.
Any $O$-lattice in a left $D$-vector space of dimension $m$ is
isomorphic to $O^m$.  Next, suppose that $D=\Mat_2(F)$. Then any
maximal order is conjugate to $M_2(R)$. An order $O$ is Eichler if and
only if there exists a nonnegative integer $n\geq 0$ such that $O$ is
conjugate to $O_n:=
  \begin{bmatrix}
    R & R\\ \pi^nR& R
  \end{bmatrix}$. Up to isomorphism, any $O_n$-lattice in the left
  $\Mat_2(F)$-module $F^{2m}$ has the
   form  
$\bigoplus_{i=1}^w\begin{bmatrix}
  R\\\pi^{t_i}R
\end{bmatrix}^{s_i}$, where 
\begin{equation}\label{eq:44}
  \sum_{i=1}^w s_i=m\quad \text{and}\quad n\geq
t_1>\cdots>t_w\geq 0.
\end{equation}
 This gives rise to a bijection between the isomorphism
classes of $O_n$-lattices in $F^{2m}$ and the pairs of tuples
$(\underline{s}, \underline{t})\in \bbN^w\times \bbZ_{\geq 0}^w$ (for some $w$
between $1$ and $m$)
with $\underline{s}=(s_1, \ldots, s_w)$ and
$\underline{t}=(t_1, \ldots, t_w)$  satisfying (\ref{eq:44}).  We
refer to \cite[Chapter~II, \S2]{vigneras} and \cite[Chapter~1, \S1.2]{Alsina-Bayer} for more details on Eichler
orders.

\begin{rem}\label{rem:assumption}
We explain what we mean by ``a mild condition on $p$" at the beginning Section~\ref{sec:stat-main-results}. For ease of exposition of the present paper, we will work out the
calculation of each $o(\uln)$ in Theorem~\ref{card.1} under the
assumption that $A_{\uln, p}=A_\uln\otimes_\Z\Z_p$ is an \'etale
$\Z_p$-algebra.  If $r=1$, this simply requires $p$ to be unramified
in $K_n$.  For $r=2$, an equivalent but more concrete assumption on
$p$ is made at the beginning of Section~\ref{sec:non-elem-case}.  The
purpose for this assumption is so that
$\scrA_{\uln, p}=A_{\uln, p}\otimes_{\Z_p}\calO_p$ is a product of
Eichler orders (cf. \cite[Lemma~2.11]{li-xue-yu:unit-gp}), and 
the proofs of Propositions~\ref{quadratic_case}, \ref{quartic_case},
and \ref{prop:p-lattices} rely on this result (also compare with
\cite[Lemma~2.12]{li-xue-yu:unit-gp} for the case where $p$ is
ramified in $K_n$). Note that the assumption holds automatically when
$p\geq 7$ so it rules out at most $p=2,3,5$.
\end{rem}

The remaining part of the paper is organized as
follows. Section~\ref{sec:elem-case} treats the \textit{isotypic} case
where $r=1$.  The \emph{non-isotypic} case where $r=2$ is treated in
Section~\ref{sec:non-elem-case}. The calculation of class numbers of
certain complicated orders arising in Section~\ref{sec:non-elem-case}
is postponed to Section~\ref{sec:class-numb-ord}.  The handful of
cases where the assumption fails will be treated in an upcoming paper
\cite{xue-yu-zheng:spIII}, where the ramification requires much
greater care.

\section{Computations of the isotypic cases}
\label{sec:elem-case}

We keep the notation of Section~\ref{sec:stat-main-results} and put
 $D=D_{p, \infty}$,  the quaternion
$\Q$-algebra ramified exactly at the prime $p$ and $\infty$. 
The goal of
this section is to calculate the terms $o(n)$ with $n\in \{3, 4, 5, 8,
12\}$ in Theorem~\ref{card.1},
under the assumption that $p$ is \textit{unramified} in $K_n$ 
(i.e.~$p\nmid n$).  
According to (\ref{eq:43}), this covers the
isotypic case for such $p$. 
Note that $p$ splits completely in $K_n$ if and only if
$p\equiv 1\pmod{n}$.  By the discussion at the beginning of
Section~\ref{sec:card.2}, if $n\in \{5,8, 12\}$ then we further assume 
that $p\not\equiv 1\pmod{n}$, for otherwise $o(n)=0$.

The cyclotomic field $K_n$ with $n\in \{3,4,5,8,12\}$ has class number 1 by
\cite[Theorem~11.1]{Washington-cyclotomic}.  For $n\in \{3,4\}$ and
$p\equiv 1\pmod{n}$, let $\scrD_n$ denote the quaternion $K_n$-algebra
ramified exactly at the two places of $K_n$ above $p$.  Since $D^\opp$
is canonically isomorphic to $D$, we have
\begin{equation}
  \label{eq:19}
\scrK_n=K_n\otimes_\Q D=
\begin{cases}
  \scrD_n \quad &\text{if } n\in \{3, 4\} \text{ and } p\equiv 1\pmod{n},\\
  \Mat_2(K_n) \quad &\text{otherwise}.
\end{cases}
\end{equation}
The order $\scrA_n\subset \scrK_n$ is maximal at every prime
$\ell\neq p$. It is also maximal at $p$ when $n\in \{3,4\}$ and
$p\equiv 1\pmod{n}$.  Let $V\simeq D^2$ be the unique \textit{faithful}
left $\scrK_n$-module of $D$-dimension $2$ (as a right $D$-vector
space).  Then $V$ is a free $\scrK_n$-module of rank 1 if $n\in
\{3,4\}$, and a simple $\scrK_n$-module if $n\in
\{5, 8, 12\}$. By (\ref{eq:11}) and (\ref{eq:13}), we have
\begin{equation}
  \label{eq:20}
\scrE_n:=\End_{\scrK_n}(V)\simeq
  \begin{cases}
    K_n\otimes_\Q D  \quad &\text{if }  n\in \{3, 4\},\\
    K_n\quad &\text{if }  n\in \{5, 8, 12\}.\\
  \end{cases}
\end{equation}
If $n\in \{3, 4\}$, then $\scrE_n$ is a quaternion algebra over the
imaginary quadratic  field $K_n$. Hence $\scrE_n$ verifies the Eichler condition
\cite[Definition~34.3]{reiner:mo}, and
$\Nr(\scrE_n^\times)=K_n^\times$ by \cite[Theorem~III.4.1]{vigneras}. 

Let $\Lambda$ be an $\scrA_n$-lattice in $V$, and
$O_\Lambda:=\End_{\scrA_n}(\Lambda)$. The order
$O_\Lambda\subset \scrE_n$ is maximal at every prime $\ell\neq p$ by the
maximality of $\scrA_{n,\ell}$.  If $n\in \{5, 8, 12\}$, then
$A_n\subseteq O_\Lambda\subset K_n$, and hence $O_\Lambda=A_n$, which
has class number $1$. If $n\in \{3, 4\}$, then $O_\Lambda$ is an
$A_n$-order in $\scrE_n$. We claim that $h(O_\Lambda)=1$ in this case
as well. If $p$ is inert in $K_n$, then it will be shown that
$O_\Lambda$ is an Eichler order in Proposition~\ref{quadratic_case},
otherwise $O_\Lambda$ is maximal in $\scrE_n$. Thus
$h(O_\Lambda)=h(A_n)=1$ by \cite[Corollaire~III.5.7]{vigneras}. 
It follows that for all $n\in \{3,4,5,8,12\}$ and $p\nmid n$, 
\begin{equation}
  \label{eq:23}
\qquad  o(n)=\abs{\prod_\ell\scrL_\ell(n)}=\abs{\scrL_p(n)}. 
\end{equation}

For each $f\in \bbN$, let $\Q_{p^f}$ be the
unique unramified extension of degree $f$ over $\Q_p$, and  $\Z_{p^f}$
be its  ring of integers.


\begin{prop} \label{quadratic_case}
Suppose that $n\in \{3,4\}$ and $p\nmid n$. Then 
\[o(3)  = 2-\left (
    \frac{-3}{p} \right)\qquad \text{and}\qquad 
o(4)=2-\left (\frac{-4}{p} \right).\]
\end{prop}
\begin{proof}
  If $p$ splits in $K_n$, then $\scrA_n$ is a maximal order in $\scrK_n$, so
  there is a unique genus of $\scrA_n$-lattices in $V$. We have
  $o(n)=1$ by (\ref{eq:23}). 

Suppose that $p$ is inert in $K_n$. Then $e(n)=1$, and $V$ is a free
$\scrK_n$-module of rank $1$.  We have $A_{n,p}=A_n\otimes\Z_p=\Z_{p^2}$, so by
\cite[Corollaire~II.1.7]{vigneras},
\[\scrA_{n, p}=A_{n,p}\otimes_{\Z_p}\calO_p\simeq
\begin{pmatrix}
  \Z_{p^2} & \Z_{p^2}\\
  p\Z_{p^2} & \Z_{p^2}
\end{pmatrix}.\]
It follows that any $\scrA_{n, p}$-lattice $\Lambda_p\subseteq V_p$ is
isomorphic to one of the following
\[\begin{pmatrix}
  \Z_{p^2} & \Z_{p^2}\\
  p\Z_{p^2} & p\Z_{p^2}
\end{pmatrix},\quad 
\begin{pmatrix}
  \Z_{p^2} & \Z_{p^2}\\
  p\Z_{p^2}& \Z_{p^2}
\end{pmatrix},\quad 
\begin{pmatrix}
  \Z_{p^2} & \Z_{p^2}\\
  \Z_{p^2} & \Z_{p^2}
\end{pmatrix}.
\]
Correspondingly, $(O_\Lambda)_p$ is isomorphic to 
\[\Mat_2(\Z_{p^2}), \quad \begin{pmatrix}
  \Z_{p^2} & \Z_{p^2}\\
  p\Z_{p^2} & \Z_{p^2}
\end{pmatrix}, \quad \Mat_2(\Z_{p^2}),\]
which verifies the claim above (\ref{eq:23}) that $O_\Lambda$ is an Eichler order when $p$ is
inert in $K_n$.  We conclude
that $o(n)=3$ by (\ref{eq:23}).
\end{proof}

\begin{prop} \label{quartic_case}
Suppose that $n\in \{5, 8, 12\}$ and $p\nmid n$. Then the formulas for
$o(n)$ in Theorem~\ref{card.1} hold. More explicitly,
\begin{enumerate}
\item $o(n)=0$ if $p\equiv 1\pmod{n}$;
\item $o(5)=2$ if $p\equiv 2,3\pmod{5}$;
\item $o(n)=4$ in the remaining cases.
\end{enumerate}
\end{prop}
\begin{proof}
  Only part (2) and (3) need to be proved.  Suppose that
  $p\not\equiv 0, 1\pmod{n}$. Then $e(n)=2$, and $V$ is a simple
  $\scrK_n$-module.  

  If $n=5$ and $p\equiv 2,3\pmod{5}$, then
  \[A_{5,p}\simeq \Z_{p^4},\quad \text{and}\quad \scrA_{5, p}=A_{5,p}\otimes_{\Z_p}\calO_p\simeq
\begin{pmatrix}
  \Z_{p^4} & \Z_{p^4}\\
  p\Z_{p^4} & \Z_{p^4}
\end{pmatrix}.\]
Any $\scrA_{5, p}$-lattice $\Lambda_p\subseteq V_p$ is
isomorphic to $\begin{pmatrix}
  \Z_{p^4} \\ p\Z_{p^4}
\end{pmatrix}$ or $\begin{pmatrix}
  \Z_{p^4} \\ \Z_{p^4}
\end{pmatrix}$. Hence $o(5)=2$ in this case. 
 
For the remaining cases, we have 
\[\scrA_{n, p}=A_{n,p}\otimes_{\Z_p}\calO_p\simeq (\Z_{p^2}\times
\Z_{p^2})\otimes_{\Z_p} \calO_p\simeq
\begin{pmatrix}
  \Z_{p^2} & \Z_{p^2}\\
  p\Z_{p^2} & \Z_{p^2}
\end{pmatrix}\times \begin{pmatrix}
  \Z_{p^2} & \Z_{p^2}\\
  p\Z_{p^2} & \Z_{p^2}
\end{pmatrix}.\]
Every $\scrA_{n, p}$-lattice $\Lambda_p\subseteq V_p$ decomposes into
$\Lambda_p^{(1)}\oplus \Lambda_p^{(2)}$, where each $\Lambda_p^{(i)}$
is a $\begin{pmatrix}
  \Z_{p^2} & \Z_{p^2}\\
  p\Z_{p^2} & \Z_{p^2}
\end{pmatrix}$-lattice
in the simple $\Mat_2(\Q_{p^2})$-module $V_p^{(i)}\simeq (\Q_{p^2})^2$.  There
are 2 isomorphism classes of $\Lambda_p^{(i)}$ for each
$i=1,2$. Therefore, $o(n)=2^2=4$.
\end{proof}

\section{Computations of the non-isotypic cases}
\label{sec:non-elem-case}

In this section,   we calculate the values of $o(\uln)$ with
\begin{equation} 
  \label{eq:24}
\uln=(n_1, n_2)\in \{(1,2), (2,3), (2,4), (2,6), (3, 4), (3,6)\}. 
\end{equation}
According to (\ref{eq:45}),
this treats all the non-isotypic cases. 
As mentioned in Remark~\ref{rem:assumption}, 
we assume that $A_{\uln, p}=A_\uln\otimes \Z_p$ is
\'etale over $\Z_p$.  Equivalently, $p$ is assumed to satisfy the
following two conditions:
\begin{enumerate}
\item[(I)] $p$ is unramified in $K_{n_i}=\Q[T]/(\Phi_{n_i}(T))$ for both $i=1,2$;
\item[(II)] $A_{\uln, p}=\Z_p[T]/(\Phi_{n_1}(T)\Phi_{n_2}(T))$ is a
  maximal $\Z_p$-order in
  $K_{\uln,p}$. 
\end{enumerate}
This rules out at most $p=2,3$ according to Table~\ref{tab:index-An} below. There exists a \textit{faithful} left $\scrK_\uln$-module
$V\simeq D^2$  if and only if
$e(n_i)=1$ for both $i=1,2$.  Thus $o(\uln)=0$ unless $p$ is inert in
$K_{n_i}$ when
$[K_{n_i}:\Q]=2$.  
So we make further restrictions on $p$ as listed in
Table~\ref{tab:index-An}.

\subsection{General structures} 
We explore the general structures
of objects of interest such as $A_\uln$, $\scrL_p(\uln)$ and so on
for all $\uln$ in (\ref{eq:24}).  This sets up the stage for a
case-by-case calculation of $o(\uln)$ in the next subsection. 

By (\ref{eq:10}),  $V=V_{n_1}\oplus V_{n_2}$, where each
$V_{n_i}$ is a simple $\scrK_{n_i}$-module with $\dim_DV_{n_i}=1$.
Therefore, 
$\scrE_\uln:=\End_{\scrK_\uln}(V)=\End_{\scrK_{n_1}}(V_{n_1})\times \End_{\scrK_{n_2}}(V_{n_2})$,
and
\begin{equation}
  \label{eq:26}
\forall i=1,2,\qquad   \End_{\scrK_{n_i}}(V_{n_i})=
  \begin{cases}
   D \qquad &\text{if } K_{n_i}=\Q;\\
   K_{n_i}\qquad &   \text{if } [K_{n_i}:\Q]=2.
  \end{cases} 
\end{equation}

 Let $O_{K_\uln}=\Z[T]/(\Phi_{n_1}(T))\times
\Z[T]/(\Phi_{n_2}(T))$ be the maximal order of $K_\uln$. There is an exact sequence of $A_\uln$-modules 
\begin{equation}
  \label{eq:27}
  0\to A_\uln\to O_{K_\uln}\xrightarrow{\psi}
  \Z[T]/(\Phi_{n_1}(T),\Phi_{n_1}(T))\to 0, 
\end{equation}
where $\psi: (x,y)\mapsto\bar{x}-\bar{y}$.  The indices 
$[O_{K_\uln}:A_\uln]$ are listed in Table~\ref{tab:index-An}. 

\begin{table}[!htbp]
\centering
\renewcommand{\arraystretch}{1.3}
\caption{}\label{tab:index-An}
\begin{tabular}{|c|c|c|c|c|}  \hline
$\uln$      & $K_\uln=K_{n_1}\times K_{n_2}$    & $[O_{K_\uln}:A_\uln]$ & $\scrE_\uln$ 
& Conditions on $p$ \\ \hline
$(1,2)$   & $\Q\times\Q $ & $2$ & $D\times D$ & $p\neq 2$ \\ \hline 

$(2,3)$ & $\Q\times \Q(\sqrt{-3})$ & $1$ &
$D\times\Q(\sqrt{-3})$ & $p\equiv 2\ (3)$ \\ \hline

$(2,4)$   & $\Q \times\Q(\sqrt{-1})$ & 
$2$ & $D \times\Q(\sqrt{-1})$ & $p\equiv 3\ (4)$ \\ \hline

$(2,6)$   & $\Q \times\Q(\sqrt{-3})$ & 
$3$ & $D \times\Q(\sqrt{-3})$ & $p\equiv 2\ (3)$ \\ \hline 

$(3,4)$ &  $\Q(\sqrt{-3})\times\Q(\sqrt{-1})$ &
$1$ &  $\Q(\sqrt{-3})\times\Q(\sqrt{-1})$ & $p\equiv 11 \ (12)$ \\
  \hline

$(3,6)$   & $\Q(\sqrt{-3})\times\Q(\sqrt{-3})$ & 
$4$ & $\Q(\sqrt{-3})\times\Q(\sqrt{-3})$ & $p\equiv 2\ (3),\ p\neq 2$ \\ \hline
\end{tabular}
\end{table} 

Observe that $A_{(2,3)}$ and $A_{(3,4)}$ are maximal orders. Let $\grp_2$ (resp.~$\grq_2$) be the unique dyadic prime of $A_4$ (resp.~$A_3$), and $\grp_3$ be the unique prime ideal of $A_6$ above $3$.  Then $A_4/\grp_2\simeq \F_2$ and $A_6/\grp_3\simeq \F_3$. 
We write down the non-maximal orders $A_\uln$ explicitly using (\ref{eq:27}):
\begin{align}
  A_{(1,2)}&=\{(a,b)\in \Z\times \Z\mid a\equiv b\pmod{2}\};\label{eq:33}\\
  A_{(2, 4)}&=\{(a,b)\in \Z\times A_{4}\mid a\equiv
  b\pmod{\grp_2}\};\label{eq:34}\\  
  A_{(2, 6)}&=\{(a,b)\in \Z\times A_{6}\mid a\equiv
  b\pmod{\grp_3}\};\label{eq:345}\\ 
  A_{(3,6)}&\simeq\{(a,b)\in A_3\times A_3\mid a\equiv
  b\pmod{\grq_2}\},\label{eq:35}
\end{align}
where $A_6=\Z[T]/(T^2-T+1)$ is identified with $A_3=\Z[T]/(T^2+T+1)$
via a change of variable $T\mapsto -T$. Applying 
\cite[Lemma~7.2]{xue-yang-yu:sp_as} if necessary, we have 
\begin{equation}
  \label{eq:28}
  h(A_{3,4})=h(A_{3,6})=1. 
\end{equation}
Recall that the class number of $\calO$ is given by 
\begin{equation}
\label{eq:29}
      h(\calO)=\frac{p-1}{12}+\frac{1}{3}\left
   (1-\left(\frac{-3}{p}\right )\right )+\frac{1}{4}\left
   (1-\left(\frac{-4}{p}\right )\right ).
\end{equation}

By our assumptions, the order $\scrA_\uln$ is non-maximal at a prime
$\ell\in \bbN$ if and only if one of the following mutually exclusive
conditions holds: 
\begin{enumerate}[(i)]
\item $\ell=p$ and $\uln\neq (1,2)$; 
\item $\ell\mid [O_{K_\uln}:A_\uln]$.
\end{enumerate}

\begin{prop}\label{prop:p-lattices}
Let $\uln=(n_1, n_2)$ be a pair in (\ref{eq:24}), and $p\in \bbN$ a prime
satisfying the corresponding condition in
Table~\ref{tab:index-An}. Then \[\abs{\scrL_p(\uln)}=[K_{n_1}:\Q][K_{n_2}:\Q]. \]
For every $\scrA_\uln$-lattice $\Lambda\subset V$, the endomorphism
ring $O_\Lambda=\End_{\scrA_\uln}(\Lambda)$ is maximal at $p$.
\end{prop}
\begin{proof}
  By assumption (II), $A_{\uln, p}=A_{n_1,p}\times A_{n_2,
    p}$. Consequently, $\Lambda_p$ decomposes as $\Lambda_{n_1,p}\oplus
  \Lambda_{n_2,p}$, where each $\Lambda_{n_i,p}$ is an $\scrA_{n_i, p}$-lattice in
  the simple $\scrK_{n_i,p}$-module $V_{n_i,p}$. It is enough to show
  that the number of isomorphism classes of $\scrA_{n_i, p}$-lattices in
  $V_{n_i,p}$ is $[K_{n_i}:\Q]$, and $\End_{\scrA_{n_i,
      p}}(\Lambda_{n_i,p})$ is maximal for each $i=1,2$.

  If $K_{n_i}=\Q$, then $\scrA_{n_i, p}=\calO_p$. We have
  $\Lambda_{n_i,p}\simeq \calO_p$, and
  $\End_{\scrA_{n_i, p}}(\Lambda_{n_i,p})=\calO_p$.

If $[K_{n_i}:\Q]=2$, then $A_{n_i,p}\simeq \Z_{p^2}$ since
  $p$ is inert in $K_{n_i}$ by our assumption. It follows that
  $\scrA_{n_i, p}=\Z_{p^2}\otimes_{\Z_p} \calO_p\simeq
  \begin{pmatrix}
    \Z_{p^2} & \Z_{p^2}\\
    p\Z_{p^2}& \Z_{p^2}
  \end{pmatrix}$, and  $\Lambda_{n_i,p}$ is isomorphic to either $  \begin{pmatrix}
    \Z_{p^2} \\
    p\Z_{p^2}
  \end{pmatrix}$ or   $\begin{pmatrix}
    \Z_{p^2} \\
   \Z_{p^2}
  \end{pmatrix}$. In both cases,  $\End_{\scrA_{n_i, p}}(\Lambda_{n_i,p})=\Z_{p^2}$.
\end{proof}

Next, we consider the other class of primes at which $\scrA_\uln$ is
non-maximal, namely the prime divisors of $[O_{K_\uln}:A_\uln]$.  According to Table~\ref{tab:index-An}, there exists a
prime $\ell$ dividing $[O_{K_\uln}:A_\uln]$ only if 
\begin{equation}\label{eq:31}
\uln\in \{(1,2), (2,4), (2,6), (3,6)\}, 
\end{equation}
and each $\uln$ above determines uniquely such a  prime $\ell$. Since
$\ell\neq p$ by our assumption, we have
$\calO_{\ell}\simeq \Mat_2(\Z_{\ell})$. By
Remark~\ref{rem:simplifications}(ii), the classification of isomorphism
classes of $\scrA_{\uln,\ell}$-lattices in $V_{\ell}$ reduces to that of
$A_{\uln, \ell}$-lattices in the $K_{\uln, \ell}$-module $V'_{\ell}$, where $V_{\ell}=(V'_{\ell})^2$. The value of $\ell$ and the structure of $V'_\ell$ for each $\uln$ is given by the following table. 

\begin{table}[htbp]
\renewcommand{\arraystretch}{1.3}
\centering
\caption{}\label{tab:V-ell}
\begin{tabular}{*{3}{|>{$}c<{$}}|}
\hline
\uln & \ell &  V'_\ell\\
\hline
(1,2) & 2 & (K_{\uln, \ell})^2=(\Q_2\times \Q_2)^2\\
\hline
(2,4) & 2 & (K_{2,\ell})^2\times K_{4,\ell}=\Q_2^2\times K_{4, 2}\\
\hline
(2,6) & 3 & (K_{2,\ell})^2\times K_{6,\ell}=\Q_3^2\times K_{6, 3}\\
\hline
(3,6) & 2 & K_{\uln, \ell}=\Q_4\times \Q_4\\
\hline
\end{tabular}
\end{table}


To classify the isomorphism classes of
$A_{\uln, \ell}$-lattices in $V_\ell'$ in each of the above cases, we apply
the theory of \emph{Bass orders}. Recall that a Bass order is a
Gorenstein order for which every order containing it in the
ambient algebra is Gorenstein
as well \cite[Section 37]{curtis-reiner:1}.  We provide a
couple of equivalent characterizations in the commutative case. Let
$R$ be a Dedekind domain with fractional field $F$, and $B$ be an
$R$-order in a finite dimensional separable semisimple $F$-algebra
$E$. Denote the maximal $R$-order in $E$ by $O_E$. The following are
equivalent:
\begin{enumerate}
\item[(i)] $B$ is a Bass order, i.e. every $R$-order $B'$ with
  $B\subseteq B'\subseteq O_E$ is Gorenstein;
\item[(ii)] every ideal $I$ of $B$ can be generated by  two elements;
\item[(iii)] the quotient $O_E/B$ is a cyclic $B$-module. 
\end{enumerate}
Characterization (ii) above is due to Bass
\cite[\S7]{Bass-MathZ-1963}, and (iii) is due to Borevich and Faddeev
(see \cite[\S37, p.789]{curtis-reiner:1}). Thanks to (iii) and (\ref{eq:27}),
$A_\uln$ is a Bass order for any arbitrary\footnote{However, the same 
does not hold 
for $\uln\in \brN^r$ with $r\geq 3$, because $A_{(1,2,4)}=\Z[T]/(T^4-1)$
  already provides a counterexample. }
$\uln\in \brN^2$. The Bass property is local in
the sense that $B$ is Bass if and only if $B_\grp$ is Bass
for every nonzero prime $\grp\subseteq R$.  In particular, $A_{\uln,
  \ell}$ is Bass for all $\uln$ in (\ref{eq:31}) and the corresponding
prime $\ell$ determined by $\uln$. 

Note that $V_\ell'$ is a free $K_{\uln, \ell}$-module when
$\uln=(1,2)$ or $(3,6)$. Let $B$ be a commutative 
Bass order and $M$ be a $B$-lattice in
a \emph{free} $E$-module of rank $m$. It follows from the result of
Borevich and Faddeev (ibid.) that there exists an ascending chain of
$R$-orders
\begin{equation}
  \label{eq:46}
B\subseteq   B_1\subseteq \cdots \subseteq B_m
\end{equation}
and an invertible $B_m$-ideal $J$ such that 
\[M\simeq B_1\oplus \cdots \oplus B_{m-1}\oplus J.  \]
The chain
of orders (\ref{eq:46}) and the isomorphism class of $J$ in the Picard
group $\Pic(B_m)$ determine uniquely the $B$-isomorphism class of $M$,
and vice versa.  If $R$ is local, then $\Pic(B_m)$ is trivial,
and we have 
\begin{equation}
  \label{eq:48}
  M\simeq B_1\oplus \cdots \oplus B_m.  
\end{equation}
In this case, the chain
(\ref{eq:46}) alone forms the isomorphic invariant of $M$.  
We will  
apply this result in the proofs of Propositions~\ref{prop:3-6} and
\ref{case_o12}. 

For $\uln=(2, 2\ell)$ with $\ell\in \{2,3\}$, the $K_{\uln, \ell}$-module
$V'_\ell$ is no longer free. Nevertheless, we can use the
Krull-Schmidt-Azumaya Theorem \cite[Theorem~6.12]{curtis-reiner:1} to
write any $A_{\uln, \ell}$-lattice $\Lambda_\ell'\subset V'_\ell$
into a direct sum of indecomposable  sublattices. Every indecomposable
lattice over a commutative Bass order is isomorphic to an ideal by
\cite[\S7]{Bass-MathZ-1963} (see \cite[Theorem~37.16]{curtis-reiner:1} for the general case). 
This allows us to classify up to isomorphism all the indecomposable $A_{\uln,
  \ell}$-lattices, and hence all $A_{\uln,
  \ell}$-lattices in $V'_\ell$.  We work out this in detail in
the proof of Proposition~\ref{prop:o(2,2s)}.

\subsection{Case-by-case calculations of $o(\uln)$} We arrange the
calculations of $o(\uln)$ in the order essentially according to the complexity of
$V'_\ell$ as a $K_{\uln, \ell}$-module in Table~\ref{tab:V-ell}.  We first
treat the cases $\uln=(2, 3)$ and $\uln=(3,4)$ in
Proposition~\ref{prop:An-maximal}. The orders $A_\uln$ are already
maximal orders in $K_\uln$ for these two $\uln$, so no classification of
local lattices is needed at any prime distinct from $p$. Next, we
treat the case $\uln=(3, 6)$ in Proposition~\ref{prop:3-6}, where
$V'_\ell$ is a free $K_{\uln, \ell}$-module of rank $1$. After that,
we treat the case $\uln=(1, 2)$ in Proposition~\ref{case_o12}, where
$V'_\ell$ is a free $K_{\uln, \ell}$-module of rank $2$. In both
previous cases, we apply the result of Borevich and Faddeev on Bass
orders. Lastly, we treat the cases $\uln=(2, 2\ell)$ with $\ell\in \{2, 3\}$ in
Proposition~\ref{prop:o(2,2s)}. The $K_{\uln, \ell}$-module $V'_\ell$
is not free for these two $\uln$, so we take the Krull-Schmidt-Azumaya approach
 instead. The calculation of class numbers of
certain complicated orders (to be defined in \eqref{o12_def} and \eqref{eq:37})
is postponed to Section~\ref{sec:class-numb-ord}.

\begin{prop}\label{prop:An-maximal}
(1) $o(2,3)=\left
  (1-\left(\frac{-3}{p}\right )\right )\, h(\calO)$ for all $p\neq
3$. \\
(2) $o(3,4)=\left
  (1-\left(\frac{-3}{p}\right )\right )\left 
  (1-\left(\frac{-4}{p}\right )\right )$   for all $p\neq 2,3$. 
\end{prop}
\begin{proof}
  Suppose that $\uln\in \{(2,3), (3,4)\}$, and $p$ satisfies the
  corresponding condition in Table~\ref{tab:index-An}.  We have
  $A_\uln=O_{K_\uln}$, so $\scrA_\uln$ is maximal at every prime
  $\ell\neq p$.  
  The
  endomorphism rings of $\scrA_\uln$-lattices in $V$ are maximal
  orders in $\End_{\scrK_\uln}(V)$, which share the same class number.
  It follows that $o(\uln)=\abs{\scrL_p(\uln)}h(O_\Lambda)$ for any
  $\scrA_\uln$-lattice $\Lambda\subset V$.  If $\uln=(2,3)$, then
  $\End_{\scrK_\uln}(V)=D\times K_3$, and
  $h(O_\Lambda)=h(\calO)h(A_3)=h(\calO)$. 
  By Proposition~\ref{prop:p-lattices}, we get $o(2,3)=2 h(\calO)$.
  If $\uln=(3,4)$, then
  $\End_{\scrK_\uln}(V)=K_3\times K_4$, and $O_\Lambda=A_3\times
  A_4=A_{(3,4)}$, which has class number 1 as remarked in (\ref{eq:28}). 
  By Proposition~\ref{prop:p-lattices}, we get $o(3,4)=4$. 
  
  For the remaining primes $p$ considered in the proposition, both sides of the formulas
  are zero. The proposition is proved.
\end{proof}

\begin{prop}\label{prop:3-6}
$o(3,6)=2 \left
  (1-\left(\frac{-3}{p}\right )\right)^2$ for all $p\neq 2,3$. 
\end{prop}
\begin{proof}
Assume that $p\neq 2, 3$.   Only the case $p\equiv 2\pmod{3}$ requires a proof.  For
  $\uln=(3,6)$, 
  $O_{K_{\uln, 2}}$ is the only order in $K_{\uln,2}$ properly
  containing $A_{\uln,2}$ by (\ref{eq:35}). So any
  $A_{\uln, 2}$-lattice $\Lambda_2'$ in $V_2'\simeq K_{\uln, 2}$ is isomorphic to
  $A_{\uln, 2}$ or $O_{K_{\uln,2}}$ by (\ref{eq:48}). Correspondingly,
  \begin{equation}
    \label{eq:32}
    \End_{A_{\uln, 2}}(\Lambda_2')=
    \begin{cases}
      A_{\uln, 2}\quad &\text{if } \Lambda_2'\simeq A_{\uln, 2},\\
     O_{K_{\uln,2}}\quad &\text{if } \Lambda_2'
     \simeq O_{K_{\uln, 2}},
    \end{cases}
  \end{equation}
and the same holds for $\End_{\scrA_{\uln, 2}}(\Lambda_2)$ by
Remark~\ref{rem:simplifications}(ii). 
  It follows from  Proposition~\ref{prop:p-lattices} that  
\[\End_{\scrA_\uln}(\Lambda)=
\begin{cases}
  A_\uln\quad &\text{if } \Lambda_2\simeq (A_{\uln, 2})^2, \\
O_{K_\uln}\quad &\text{if } \Lambda_2\simeq (O_{K_{\uln, 2}})^2,
\end{cases}
\]
for  any $\scrA_\uln$-lattice
  $\Lambda\subset V$. 
Recall that $h(A_\uln)=h(O_{K_\uln})=1$ by (\ref{eq:28}). Therefore,
when $\uln=(3,6)$, $p\equiv 2\pmod{3}$ and $p\neq 2$,  we have 
\[o(\uln)=\abs{\scrL_2(\uln)}\cdot\abs{\scrL_p(\uln)}=2\cdot 4=2 \left
  (1-\left(\frac{-3}{p}\right )\right)^2.\qedhere\]
\end{proof}

Now suppose that $\uln=(1,2)$.  Then $K_\uln=\Q\times \Q$, and
$A_\uln$ is the unique suborder of index 2 in
$O_{K_\uln}=\Z\times \Z$. To write down the formula for $o(1,2)$,
we define a few auxiliary orders.  Let 
$\bbO_1(1,2):=\calO \times \calO$,  a maximal order in
$\End_{\scrK_\uln}(V)= D \times D$.  Fix an isomorphism
$\calO_2\simeq \Mat_2(\Z_2)$, and thereupon an isomorphism
\[ \bbO_1(1,2)_2=(\calO\times\calO)\otimes \Z_2\simeq
\Mat_2(\Z_2\times\Z_2) =\Mat_2(O_{K_{\uln, 2}}). \]
Let $\bbO_8(1,2)$ and $\bbO_{16}(1,2)$ be the suborders of
$\bbO_1(1,2)$ of index $8$ and $16$ respectively such that
\begin{equation} \label{o12_def}
\begin{split}
  \bbO_{8}(1,2)_2 &=
\begin{pmatrix}
  A_{\uln, 2} & 2 O_{K_{\uln, 2}} \\  O_{K_{\uln,2}} & O_{K_{\uln, 2}}\\
\end{pmatrix},\qquad \mathbb{O}_{16}(1,2)_2=\Mat_2(A_{\uln,2}); \\  
\bbO_i(1,2)_{\ell'}&=\bbO_1(1,2)_{\ell'}\qquad \forall \text{ prime }
\ell'\neq 2 \text{ and } i=8, 16. 
\end{split}
\end{equation}

\begin{prop} \label{case_o12}
If $p=3$, then $o(1,2)=3$. For $p \neq 2,3$,  we have
\begin{equation}\label{eq:o12}
  \begin{split}
o(1,2) = &h(\bbO_1(1,2))+h(\bbO_8(1,2))+h(\bbO_{16}(1,2))\\
=& \frac{(p-1)^2}{9} + \frac{p+15}{18}\left(1-\left (
    \frac{-3}{p} \right)\right) + \frac{p+2}{6}\left(1-\left (
    \frac{-4}{p} \right)\right) \\  
 &+ \frac{1}{6}\left(1-\left ( \frac{-3}{p} \right
   )\right)\left(1-\left ( \frac{-4}{p} \right )\right).    
  \end{split}
\end{equation}
\end{prop}
\begin{proof}
Throughout this proof, we assume that $p\neq 2$.   By Table~\ref{tab:V-ell}, $V_2'$ is a free $K_{\uln,2}$-module of
  rank $2$. According to (\ref{eq:48}), any $A_{\uln, 2}$-lattice $\Lambda_2'\subseteq V_2'$ is
  isomorphic to $A_{\uln, 2}^{j} \oplus (O_{K_{\uln, 2}})^{2-j}$ with $j=0,1,2$.
Correspondingly, the endomorphism ring  $\End_{A_{\uln,
    2}}(\Lambda_2')$ is isomorphic to 
\[\bbO_1(1,2)_2,\qquad \bbO_8(1,2)_2,\qquad \bbO_{16}(1,2)_2.\]
Since $\abs{\scrL_p(\uln)}=1$ by Proposition~\ref{prop:p-lattices},
there are three genera of $\scrA_\uln$-lattices in $V$. Each is
represented by a lattice with endomorphism ring $\bbO_i(1,2)$ for
$i\in \{1,8, 16\}$, respectively. It follows that 
\begin{equation}
  \label{eq:36}
o(1,2)=h(\bbO_1(1,2))+h(\bbO_8(1,2))+h(\bbO_{16}(1,2)). 
\end{equation}
The class number $h(\bbO_8(1,2))$ is given by
Proposition~\ref{prop:class-num-O8},  and $h(\bbO_{16}(1,2))$ is given
by Proposition~\ref{prop:O16-calc}. 
 Lastly, we have $h(\bbO_1(1,2))=h(\calO)^2$. 
The explicit formula
for $o(1,2)$ follows from (\ref{eq:36}). 
\end{proof}

Finally, we study the terms $o(2,2\ell)$ for $\ell\in \{2,3\}$.  We have
$[O_{K_\uln}:A_\uln]=\ell$, and $\End_{\scrK_\uln}(V)=D\times K_{2\ell}$ by
(\ref{eq:26}).  Let $\bbO_1(2,2\ell)$ be the maximal order
$\calO\times A_{2\ell}\subset \End_{\scrK_\uln}(V)$. Recall that
$p\neq \ell$ by our assumption, so we fix an isomorphism
$\calO_\ell\simeq \Mat_2(\Z_\ell)$. By an abuse of notation, we still write
$\grp_\ell$ for the unique prime ideal of $A_{2\ell,\ell}$
above $\ell$.  Let $\bbO_{\ell^2}(2,2\ell)$ be the suborder of index $\ell^2$ in
$\bbO_1(2,2\ell)$ such that
\begin{equation}
  \label{eq:37}
  \begin{aligned}
 \bbO_{\ell^2}(2,2\ell)_\ell&=\left\{\left (
  \begin{bmatrix}
    a_{11} & a_{12}\\
    a_{21} & a_{22}
  \end{bmatrix}, b\right)\in \bbO_1(2,2\ell)_\ell\middle \vert 
\begin{aligned}
  a_{21}\equiv 0 &\pmod{\ell}\\
a_{22} \equiv b &\pmod{\grp_\ell}
\end{aligned}
\right\};  \\
\bbO_{\ell^2}(2,2\ell)_{\ell'}&=\bbO_1(2,2\ell)_{\ell'}\qquad \forall \text{ prime }\ell'\neq \ell. 
  \end{aligned}
\end{equation}

\begin{prop} \label{prop:o(2,2s)} 
Suppose that $\ell\in\{2,3\}$ and $p$ satisfies the corresponding
condition for $\uln=(2,2\ell)$ in
Table~\ref{tab:index-An}. Then 
$o(2,2\ell)=2h(\bbO_1(2,2\ell)+2h(\bbO_{\ell^2}(2,2\ell))$. 
 More explicitly, 
\begin{align*}\label{eq:o13}
o(2,4) &=\left( \frac{p+3}{3}-\frac{1}{3}\left (
      \frac{-3}{p} \right)\right)\left(1-\left ( \frac{-4}{p}
    \right)\right) \quad \text{if } p\neq 2; 
  \\
o(2,6) &= \left(\frac{5p+18}{12}+\frac{1}{3}\left (
       \frac{-3}{p}\right)-\frac{1}{4}\left(\frac{-4}{p}\right)\right)\left(1-\left (
       \frac{-3}{p} \right)\right)           \quad \text{if } p\neq
         3. 
\end{align*}
\end{prop}
\begin{proof}
  For the explicit formulas for $o(2,4)$ and $o(2,6)$, only the cases
  where $p$ satisfies the corresponding condition in
  Table~\ref{tab:index-An} are nontrivial and need to be proved.  

  Let $V_\ell'=\Q_\ell^2\times K_{2\ell,\ell}=\Q_\ell\oplus K_{\uln,\ell}$ be the module
  over $K_{\uln, \ell}=\Q_\ell\times K_{2\ell, \ell}$ in Table~\ref{tab:V-ell}. We claim
  that any $A_{\uln, \ell}$-lattice $\Lambda_\ell'\subset V_\ell'$ is
  isomorphic to $\Sigma_0:=\Z_\ell\oplus O_{K_{\uln,\ell}}$ or
  $\Sigma:=\Z_\ell\oplus A_{\uln, \ell}$. By the Krull-Schmidt-Azumaya Theorem \cite[Theorem~6.12]{curtis-reiner:1}, every $A_{\uln, \ell}$-lattice is uniquely expressible as a finite direct sum of indecomposable sublattices, up to isomorphism and order of occurrence of the summands.  Recall that any indecomposable
  lattice over a Bass order is isomorphic to an ideal
  \cite[\S7]{Bass-MathZ-1963}. Let $I_\ell$ be an 
  $A_{\uln, \ell}$-ideal. Then $I_\ell\otimes_{\Z_\ell}\Q_\ell$ is isomorphic to
  $\Q_\ell$, $K_{2\ell, \ell}$ or $K_{\uln, \ell}$. If $I_\ell\otimes_{\Z_\ell}\Q_\ell\simeq
      K_{\uln, \ell}$, then the result of Borevich and Faddeev
      \cite[\S37, p.789]{curtis-reiner:1} implies
      that $I_\ell$ is isomorphic to either $A_{\uln, \ell}$ or
      $O_{K_{\uln, \ell}}=\Z_\ell\oplus A_{2\ell, \ell}$. Clearly, $O_{K_{\uln, \ell}}$ is
      decomposable. Therefore, if $I_\ell$ is indecomposable, then we have 
  \begin{equation}
    \label{eq:49}
    I_\ell\simeq 
    \begin{cases}
      \Z_\ell \qquad &\text{if } I_\ell\otimes_{\Z_\ell}\Q_\ell\simeq \Q_\ell;\\
      A_{2\ell, \ell}\qquad &\text{if } I_\ell\otimes_{\Z_\ell}\Q_\ell\simeq K_{2\ell, \ell};\\
      A_{\uln, \ell}\qquad &\text{if } I_\ell\otimes_{\Z_\ell}\Q_\ell\simeq
      K_{\uln, \ell}. 
    \end{cases}
  \end{equation}
Write $\Lambda'_\ell=\Z_\ell^{t_1}\oplus  A_{2\ell, \ell}^{t_2}\oplus A_{\uln,
  \ell}^{t_3}$. Since $\Lambda'_\ell\otimes_{\Z_\ell}\Q_\ell\simeq
\Q_\ell^2\times K_{2\ell, \ell}$, we have $(t_1, t_2, t_3)=(2, 1, 0)$ or
$(1, 0, 1)$. The claim is verified.

  Direct calculation shows that
\[\End_{A_{\uln, \ell}}(\Lambda_\ell')=
\begin{cases}
  \bbO_1(2, 2\ell)_\ell  &\qquad \text{if } \Lambda_\ell'\simeq \Sigma_0;\\
  \bbO_{\ell^2}(2, 2\ell)_\ell  &\qquad \text{if } \Lambda_\ell'\simeq \Sigma.\\
\end{cases}
\] 
The classification at $\ell$ partitions the set of isomorphism
classes of $\scrA_\uln$-lattices $\Lambda\subset V$  into two subsets, 
according to the local isomorphism classes of $\Lambda_\ell$. Each subset consists
of two genera by Proposition~\ref{prop:p-lattices}. Taking into
account of the maximality of $\End_{\scrA_{\uln, p}}(\Lambda_p)$ for
every $\Lambda$, we have 
\begin{equation}\label{eq:o22s}
o(2, 2\ell)=2h(\bbO_1(2,2\ell))+2h(\bbO_{\ell^2}(2,2\ell)).
\end{equation}
The class numbers of $\bbO_4(2,4)$ and $\bbO_9(2,6)$ are calculated in
Proposition~\ref{prop:o144_o139}, and $h(\bbO_1(2,2\ell))=h(\calO)h(A_{2\ell})=h(\calO)$ for $\ell\in \{2,3\}$. 
The explicit formulas for $o(2,4)$ and $o(2,6)$ follow directly. 
\end{proof}

\begin{rem}
  When $\uln=(2,6)$, $A_\uln\simeq A_{(1,3)}=\Z[T]/(T^3-1)$ coincides
  with the group ring $\Z[C_3]$ for the cyclic group $C_3$ of order
  $3$. The classification of $A_{\uln,3}$-lattices is equivalent to
  that of $\Z_3$-representations of $C_3$. Similarly, $A_{(2,4)}$ is a
  quotient of $\Z[C_4]$.  Therefore, one may also apply the result of
  Heller and Reiner \cite{Heller-Reiner-1962} on indecomposable
  integral representations over cyclic groups of order $\wp^2$ ($\wp\in\bbN$ a
  prime) to obtain the claim in Proposition~\ref{prop:o(2,2s)}.
\end{rem}

\section{Class numbers of certain orders} 
\label{sec:class-numb-ord}

In this section, we compute the class numbers of the orders $\bbO_8(1,2)$, $\bbO_{16}(1,2)$, $\bbO_4(2,4)$, and $\bbO_9(2,6)$,  defined in
(\ref{o12_def}) and (\ref{eq:37}). Throughout this section, the prime
$p$ is assumed to satisfy the corresponding 
 condition in Table~\ref{tab:index-An} for $\uln=(1,2), (2,4), (2,6)$ respectively. We first work out $h(\bbO_4(2,4))$ and $h(\bbO_9(2,6))$ in Proposition~\ref{prop:o144_o139}, and then  $h(\O_8(1,2))$ in Proposition~\ref{prop:class-num-O8}, and lastly $h(\bbO_{16}(1,2))$ in Proposition~\ref{prop:O16-calc}. 

We recall some
properties of ideal classes in more general settings. 
Let $\calR\subset\calS$ be two $\Z$-orders in a finite dimensional
semisimple $\Q$-algebra $\calB$. There is a natural
\textit{surjective} map between the sets of locally principal right
ideal classes
\[\pi:\Cl(\calR)\to \Cl(\calS), \qquad [I]\mapsto [I\calS]. \] 
The surjectivity is best seen using the adelic language, where $\pi$
is given by 
\begin{equation}\label{eq:40}
  \pi: \calB^\times\backslash\wh \calB^\times/\wh \calR^\times \to
\calB^\times\backslash\wh \calB^\times/\wh \calS^\times, \qquad
\calB^\times x\wh 
\calR^\times \mapsto  \calB^\times x\wh\calS^\times,\quad \forall x\in \wh\calB^\times.     
\end{equation}

Let $J\subset \calB$ be a locally principal right $\calS$-ideal.  We
study the fiber $\pi^{-1}([J])$. Write $\wh J=x\wh \calS$ for some
$x\in \wh\calB^\times$, and set
$\calS_J:=O_l(J)=\calB\cap x\wh\calS x^{-1}$, the associated left
order of $J$. By (\ref{eq:40}), we have 
\begin{equation}
  \label{eq:1}
\pi^{-1}([J])= \pi^{-1}(\calB^\times x \wh \calS^\times)=\calB^\times \backslash
  (\calB^\times 
x\wh \calS^\times)/ \wh\calR^\times.
\end{equation}
 Multiplying $\calB^\times
x\wh \calS^\times$ from the left by $x^{-1}$ induces a 
bijection 
\[\calB^\times \backslash (\calB^\times
x\wh \calS^\times)/ \wh\calR^\times\simeq (x^{-1}\calB^\times x) \backslash (x^{-1}\calB^\times x\wh
  \calS^\times) / \wh\calR^\times, \] 
and the latter is in turn isomorphic to
$ (x^{-1}\calB^\times x\cap \wh \calS^\times) \backslash \wh
\calS^\times/ \wh\calR^\times$. Therefore, we  obtain a double
coset description of the fiber
\begin{equation}\label{eq:41}
  \pi^{-1}([J])\simeq (x^{-1}\calS_J^\times x)\backslash\wh
\calS^\times/ \wh\calR^\times.
\end{equation}

\begin{lem}\label{lem:fiber-class-map}
  Suppose that
  $\wh \calS^\times \subseteq \calN(\wh\calR)$, the normalizer of
  $\wh\calR$ in $\wh\calB^\times$. Then the suborder
  $\calR_J:=x\wh \calR x^{-1}\cap \calB$ of $\calS_J$ is independent of the choice
  of $x\in \wh\calB^\times$ for $J$, and 
\[\abs{\pi^{-1}([J])}=\frac{[\wh \calS^\times : \wh
  \calR^\times]}{[\calS_J^\times:\calR_J^\times]}.\] 
\end{lem}
\begin{proof}
Suppose that $\wh J=x'\wh\calS$ for $x'\in
\wh\calB^\times$ as well. Then there exists $u\in
  \wh 
  \calS^\times$ such that $x'=xu$. Since $\wh\calS^\times\subseteq
  \calN(\wh \calR)$,  we have 
\[x'\wh \calR x'^{-1}\cap \calB=xu\wh \calR
  u^{-1}x^{-1}\cap 
\calB= x\wh \calR x^{-1}\cap \calB=\calR_J\subset \calS_J, \]
which proves the independence of $\calR_J$ of the choice of
$x$.  If $I$ is a locally principal right $\calR$-ideal such that
    $I\calS=J$, then $\calR_J=O_l(I)$, the associated left order of
    $I$. Conjugating by $x\in \wh\calB^\times$ on the right hand side
    of (\ref{eq:41}), we obtain 
    \begin{equation}
      \label{eq:42}
  \pi^{-1}([J])\simeq \calS_J^\times\backslash
(x\wh\calS^\times x^{-1})/ (x\wh\calR^\times
x^{-1})=\calS_J^\times\backslash \wh\calS_J^\times/\wh\calR_J^\times.
    \end{equation}


    The assumption $\wh\calS^\times\subseteq \calN(\wh\calR)$ also
    implies that $\wh \calR^\times\unlhd\wh \calS^\times$, and hence
    $\wh \calR_J^\times\unlhd\wh\calS_J^\times$ and
    $\calR_J^\times\unlhd \calS_J^\times$. The left action of
    $S_J^\times$ on the quotient group
    $\wh\calS_J^\times/ \wh\calR_J^\times$ factors through
    $\calS_J^\times /\calR_J^\times\subseteq \wh\calS_J^\times/\wh\calR_J^\times$, and its
    orbits are the right cosets of
    $\calS_J^\times/\calR_J^\times$ in
    $\wh\calS_J^\times/ \wh\calR_J^\times$. Thus
\[\abs{\pi^{-1}([J])}=[\wh \calS_J^\times :
\wh \calR_J^\times]/[\calS_J^\times:\calR_J^\times]=[\wh \calS^\times :
\wh \calR^\times]/[\calS_J^\times:\calR_J^\times].\qedhere\]
\end{proof}

\begin{rem}  The condition   $\wh \calS^\times \subseteq \calN(\wh\calR)$ implies
that $\wh\calR^\times \unlhd \wh\calS^\times$. However, the converse
does not hold in general. It is enough to provide a counterexample
locally at a prime $\ell$, say, $\ell=2$. Let $\calS_2:=\Mat_2(\Z_2)$, and $\calR_2=\begin{pmatrix}
    \Z_2 & 2\Z_2\\ 2\Z_2 & \Z_2
  \end{pmatrix}$, an Eichler order of level $4$ in $\calS_2$. Then \[\calR_2^\times=\left\{x\in
  \Mat_2(\Z_2)\,\middle\vert\, x\equiv
  \begin{pmatrix}
    1 & 0 \\ 0 & 1
  \end{pmatrix}
\pmod{2\calS_2}\right\}\unlhd \calS_2^\times=\GL_2(\Z_2).\] 
On the other hand, let $u=
\begin{pmatrix}
  1 & 1 \\ 0 & 1
\end{pmatrix}\in \calS_2^\times$, and $y=
\begin{pmatrix}
  1 & 0 \\ 0 & 0
\end{pmatrix}\in \calR_2$. Then 
\[uyu^{-1}=\begin{pmatrix}
  1 & 1 \\ 0 & 1
\end{pmatrix}\begin{pmatrix}
  1 & 0 \\ 0 & 0
\end{pmatrix}\begin{pmatrix}
  1 & -1 \\ 0 & 1
\end{pmatrix}=\begin{pmatrix}
  1 & -1 \\ 0 & 0
\end{pmatrix}\not\in \calR_2.\]
\end{rem}

\begin{cor}\label{lem:eq-class-num}
  Keep the notation and assumption of
  Lemma~\ref{lem:fiber-class-map}. If the natural homomorphism 
  $\calS_J^\times\to \wh \calS_J^\times/\wh\calR_J^\times$ is
  surjective for each ideal class $[J]\in \Cl(\calS)$, then $\pi$ is bijective.
\end{cor}
\begin{proof}It is enough to show that $\pi$ is injective. The
  surjectivity of $\calS_J^\times\to \wh 
  \calS_J^\times/\wh\calR_J^\times$ implies that the monomorphism
  $\calS_J^\times/\calR_J^\times\hookrightarrow \wh 
  \calS_J^\times/\wh\calR_J^\times$ is an isomorphism, and hence
  $\abs{\pi^{-1}([J])}=[\wh\calS_J^\times/\wh\calR_J^\times:\calS_J^\times/
  \calR_J^\times]=1$.   
\end{proof}

Let $D=D_{p, \infty}$ be the unique quaternion algebra over $\Q$
ramified exactly at $p$ and $\infty$, and $\calO\subset D$ a maximal order in
$D$.  Let $\ell\in \{2,3\}$,
and assume that $p\neq \ell$. Fix an isomorphism
$\calO\otimes_\Z \Z_\ell \simeq \Mat_2(\Z_\ell)$. We write $\calO^{(\ell)}$ for the Eichler order of
level $\ell$ in $\calO$ such that $\calO^{(\ell)}\otimes \Z_{\ell'}=\calO\otimes
\Z_{\ell'}$ for every prime $\ell'\neq \ell$, and 
\[\calO^{(\ell)}\otimes\Z_\ell=
  \begin{bmatrix}
    \Z_\ell & \Z_\ell\\
   \ell\Z_\ell & \Z_\ell
  \end{bmatrix}. 
\]
 The formula for $h(\calO^{(\ell)})$ is given in \cite[Theorem 16]{Pizer-1976}:
 \begin{equation} \label{class_number_Eichler}
\begin{split}
h(\calO^{(\ell)}) =& \frac{(p-1)(\ell+1)}{12} + \frac{1}{3}\left(
  1-\left(\frac{-3}{p} \right)  \right) \left( 1 + \left(\frac{-3}{\ell}
  \right)\right) \\  &+ \frac{1}{4}\left( 1-\left(\frac{-4}{p} \right)
\right) \left( 1+ \left(\frac{-4}{\ell} \right)\right), \quad \text{for
   }\ell\in
\{2,3\} \text{ and } p\neq \ell.
\end{split}
 \end{equation}
Here for $m\in \{3,4\}$, the symbol $\Lsymb{-m}{2}$ should be
understood as the Jacobi symbol $\Lsymb{\Z[\zeta_m]}{2}$
(\cite[Chapter~III, \S5, p.~94]{vigneras}), where $\zeta_m$ denotes a
primitive $m$-th root of unity.   

 \begin{prop} \label{prop:o144_o139} Suppose that $\ell\in \{2,3\}$ and
   $p\neq \ell$. 
   Let $\bbO_{\ell^2}(2,2\ell)$ be the order defined in (\ref{eq:37}). Then
  \begin{itemize}
 \item[(a)] $ h(\mathbb{O}_4(2,4))= \frac{1}{4}\left( p-\left(
       \frac{-4}{p}\right)\right)$ if $p\neq 2$;
 \item[(b)] $h(\mathbb{O}_9(2,6))= \frac{1}{3}\left( p-\left(
       \frac{-3}{p}\right) \right)$ if $p\neq 3$. 
 \end{itemize}
 \end{prop}
\begin{proof}
  For simplicity, we set $\bbO_{\ell^2}=\bbO_{\ell^2}(2, 2\ell)$, and define
  $\bbO_\ell:=\calO^{(\ell)}\times A_{2\ell}$, which contains $\bbO_{\ell^2}$ and
  is a suborder of index $\ell$ in $\bbO_1(2,2\ell)=\calO\times
  A_{2\ell}$.
  Recall that $\grp_\ell$ denotes the unique ramified prime in
  $A_{2\ell}$. We have $A_{2\ell}/\grp_\ell=\F_\ell$, and the
  canonical map
  $A_{2\ell}^\times \to
  (A_{2\ell}/\grp_\ell)^\times=\F_\ell^\times$
  is surjective.  

  It is straightforward to check that $\wh\bbO_\ell^\times\subseteq
  \calN(\wh \bbO_{\ell^2})$,  and
  $\wh\bbO_{\ell}^\times/\wh \bbO_{\ell^2}^\times\cong\F_\ell^\times$. Let
  $Z(\bbO_\ell)$ be the center of $\bbO_\ell$. Then $Z(\bbO_\ell)=\Z\times
  A_{2\ell}$, and its unit group $Z(\bbO_\ell)^\times=\{\pm 1\}\times A_{2\ell}^\times$ maps surjectively onto
  $\wh\bbO_\ell^\times/\wh\bbO_{\ell^2}^\times$.  Since
  $Z(\bbO_\ell)=Z(O_l(J))$ for every locally principal right ideal $J$ of $\bbO_\ell$, the assumptions of
  Corollary~\ref{lem:eq-class-num} are satisfied. Therefore,
\begin{equation*}
  \label{eq:2}
  h(\bbO_{\ell^2})=h(\bbO_\ell)=h(\calO^{(\ell)})h(A_{2\ell})=h(\calO^{(\ell)}),
  \qquad  \text{for } \ell=2,3.  
\end{equation*}
 Applying formula (\ref{class_number_Eichler}),  we obtain
\begin{align*}
  h(\mathbb{O}_4(2,4))& = h(\calO^{(2)})=\frac{1}{4}\left(
  p-\left( \frac{-4}{p}\right)\right)\qquad \text{if } p\neq 2;\\
  h(\mathbb{O}_9(2,6))& = h(\calO^{(3)})=\frac{1}{3}\left(
  p-\left( \frac{-3}{p}\right) \right)\qquad \text{if } p\neq 3. \qedhere
\end{align*}
\end{proof}


Next, we assume that $p\neq 2$ and  calculate the class numbers of the orders
$\mathbb{O}_8(1,2)$ and $ \mathbb{O}_{16}(1,2)$ defined in
(\ref{o12_def}).   For simplicity, let
$\mathbb{O}_i=\mathbb{O}_i(1, 2)$ for $i\in \{1, 8, 16\}$.
\begin{prop}\label{prop:class-num-O8}
Suppose that $p\neq 2$. Then 
\[  h(\mathbb{O}_8(1,2))=\frac{1}{16}\left( p-\left(
  \frac{-4}{p}\right)\right)^2.  \]
\end{prop}
\begin{proof}
By an
abuse of notation, we still write $\calO^{(2)}$ for the Eichler
order of $\calO$ of level 2 such that $\calO^{(2)}\otimes\Z_2=
  \begin{bmatrix}
    \Z_2 & 2\Z_2\\
   \Z_2 & \Z_2
 \end{bmatrix}$
 and $\calO^{(2)}\otimes \Z_{\ell'}=\calO\otimes \Z_{\ell'}$ for all primes
 $\ell'\neq 2$.   Put $\scrO_4:=\calO^{(2)} \times \calO^{(2)}$, which
 is a suborder of $\mathbb{O}_1$ of index 4 containing
 $\mathbb{O}_8$. One checks that $\wh\scrO_4^\times\subseteq \calN(\wh
 \bbO_8)$, and  $\widehat{\mathbb{O}}_8^\times =\wh\scrO_4^\times$, so
 the assumptions of Corollary~\ref{lem:eq-class-num} are automatically
 satisfied. We have
\begin{equation}
  \label{case_o12_8}
  h(\mathbb{O}_8(1,2))=h(\calO^{(2)}\times
  \calO^{(2)})=h(\calO^{(2)})^2=\frac{1}{16}\left( p-\left(
  \frac{-4}{p}\right)\right)^2.  \qedhere
\end{equation}
\end{proof}

To calculate the class number of $\mathbb{O}_{16}$, we first note that
$2\mathbb{O}_1\subset \mathbb{O}_{16}$, and the quotient ring
$\mathbb{O}_{16}/2\mathbb{O}_1\cong 
\Mat_2(\F_2)$ embeds diagonally into $\mathbb{O}_1/2\mathbb{O}_1\cong
\Mat_2(\F_2)^2$. In this case, $\widehat{\mathbb{O}}_{16}^\times$ is \textit{not}
normal in $\widehat{\mathbb{O}}_1^\times$, so $\wh\bbO_1^\times
\not\subseteq \calN(\wh\bbO_{16})$. This prevents us from applying
Lemma~\ref{lem:fiber-class-map} or Corollary~\ref{lem:eq-class-num} to the current
situation. 

We consider the natural surjective map
$\pi:\Cl(\bbO_{16})\to \Cl(\bbO_1)$ and work out explicitly the
cardinality of each fiber.  If $[J]\in \Cl(\bbO_1)$ is a
right ideal class of $\bbO_1$ with $\widehat{J} =x\wh\bbO_1$ for an
element $x\in (\wh D^\times)^2$, then by (\ref{eq:41}) one has a
bijection
\begin{equation}
  \label{eq:fiber}
  \pi^{-1}([J])\simeq x^{-1} \bbO_J^\times x \backslash \wh
  \bbO_1^\times/\wh \bbO_{16}^\times, \quad\text{where}\quad \bbO_J=O_l(J)=D^2\cap x \wh
  \bbO_1 
  x^{-1}.     
\end{equation}
If $p\neq 2, 3$, then $\bbO_J^\times \simeq C_{2 j_1} \times C_{2 j_2}$ for some
$1\le j_1,j_2\le 3$. Here $C_n$
denotes a cyclic group of order $n$. Given an arbitrary set $X$,
we write $\Delta(X)$ for the diagonal of $X^2$.  

\begin{lem} \label{sec:o16-double-coset-fiber}
(1) Let $[J]\in \Cl(\bbO_1)$ be a right ideal class of $\bbO_1$. If $\bbO_J^\times \simeq
    C_{2j_1}\times C_{2j_2}$, where  $1\le j_1,j_2\le 3$, then there
    is a bijection
    $\pi^{-1}([J])\simeq C_{j_1}\backslash S_3 /C_{j_2}$, where $S_n$
    denotes the symmetric group on $n$ letters.

(2) Let $c_{j_1,j_2}:=\abs{C_{j_1}\backslash
    S_3 /C_{j_2}}$ for $1\leq j_1,
  j_2\leq 3$. Then the values of $c_{j_1, j_2}$ are listed in the
  following table:
\begin{center}
\begin{tabular}{|c|c|c|c|}  \hline
$c_{j_1,j_2}$  & $1$  & $2$  & $3$ \\ \hline
$1$ & $6$ & $3$ & $2$  \\ \hline
$2$ & $3$ & $2$ & $1$  \\ \hline
$3$ & $2$ & $1$ & $2$  \\ \hline 
\end{tabular} 
\end{center} 
\end{lem}
\begin{proof}
  (1) We may regard $C_{2j_1}\times C_{2j_2}=x^{-1}\bbO_J^\times x$ as
  a subgroup of $\wh \bbO_1^\times$. As
  $1+2 \wh \bbO_1 \subset \wh \bbO_{16}^\times$, modulo this subgroup,
  one has
  $\wh \bbO_1^\times/\wh \bbO_{16}^\times \simeq (\GL_2(\F_2)\times
  \GL_2(\F_2))/ \Delta(\GL_2(\F_2))$.
  For any unit $\zeta\in \calO^\times$, we have either $\zeta^4=1$ or
  $\zeta^6=1$, and $\Z[\zeta]$ coincides with the ring of integers of
  $\Q(\zeta)$. By a lemma of Serre, if $\zeta$ is a root of unity
  which is congruent to $1$ modulo $2$, then $\zeta=\pm 1$.  Thus, for
  $1\leq j\leq 3$, the map $C_{2j}\to \GL_2(\F_2)$ factors through an
  embedding $C_j\simeq (C_{2j}/C_2)\hookrightarrow \GL_2(\F_2)$. Note
  that $\GL_2(\F_2)\simeq S_3$.  Since cyclic subgroups of order $j$
  of $S_3$ are conjugate, the double coset space
  $(C_{j_1}\times C_{j_2})\backslash (S_3\times S_3) /\Delta(S_3)$
  does not depend on how $C_j$ embeds into $S_3$. Every element of
  $(S_3\times S_3)/\Delta(S_3)$ is represented by a unique $(a,1)$
  with $a\in S_3$. For $(c_1,c_2)\in C_{j_1}\times C_{j_2}$, one has
  $(c_1,c_2)\cdot (a,1)=(c_1s,c_2)\sim (c_1 a c_2^{-1},1)$.  The map
  $(a,1)\mapsto a$ yields a bijection
  $(C_{j_1}\times C_{j_2}) \backslash (S_3\times
  S_3)/\Delta(S_3)\simeq C_{j_1}\backslash S_3/C_{j_2}$.
  Therefore, there is a bijection
\[ \pi^{-1}([J])\simeq (C_{j_1}\times C_{j_2})\backslash
\GL_2(\F_2)^2/\Delta(\GL_2(\F_2))\simeq C_{j_1}\backslash
S_3/C_{j_2}.\] 

(2) This is clear if one of the $j_i$ is $1$ or $3$ as $C_3$ is a normal
subgroup of $S_3$. To see $c_{2,2}=2$, one may view $C_2$ as a Borel
subgroup of $S_3=\GL_2(\F_2)$; then the result follows from 
the Bruhat decomposition. 
\end{proof}

\begin{prop}\label{prop:O16-calc}
We have 
\[  h(\bbO_{16}(1,2))=\frac{(p-1)^2}{24}+\frac{1}{4}\left(1-\left
  (\frac{-4}{p}\right)\right)+\frac{2}{3}\left(1-\left
  (\frac{-3}{p}\right)\right)\quad \text{if } p\neq 2, 3.  
 \]
Moreover, if $p=3$, then $h(\bbO_{16}(1,2))=1$. 
\end{prop}
\begin{proof}
First suppose that $p=3$.  By \cite[Proposition~V.3.1]{vigneras}, we have
  $h(\calO)=1$, and $\calO^\times/\{\pm 1\} \simeq S_3$. It follows
  that $h(\bbO_1)=h(\calO)^2=1$, and hence
  $\Cl(\bbO_{16})=\pi^{-1}([\bbO_1])\simeq
  \bbO_1^\times\backslash\wh\bbO_1^\times/\wh\bbO_{16}^\times$ by
  (\ref{eq:41}).  The same line of argument as
  that of part (1) of Lemma~\ref{sec:o16-double-coset-fiber} shows
  that $h(\bbO_{16})=\abs{(S_3)^2\backslash (S_3)^2/\Delta(S_3)}=1$. 

Next, suppose that $p\neq 2, 3$. For $n=1,2,3$, put
\begin{equation}
  \label{eq:hn}
\Cl_n(\calO):=\{[I]\in\Cl(\calO) \mid O_l(I)^\times \simeq
C_{2n}\},\quad\text{and}\quad h_n=h_n(\calO):=\abs{\Cl_n(\calO)}.
\end{equation}
By \cite[Proposition~V.3.2]{vigneras}, if $p\neq 2,3$, then 
\begin{align}
  \label{eq:6}
  h_2(\calO)&=\frac{1}{2} \left(1-\left(\frac{-4}{p}\right) \right), \quad 
  h_3(\calO)=\frac{1}{2} \left(1-\left(\frac{-3}{p}\right) \right), \\
  \label{eq:7}
  \begin{split}
   h_1(\calO)&=h(\calO)-h_2(\calO)-h_3(\calO)\\
&=\frac{p-1}{12}-\frac{1}{4} \left(1-\left(\frac{-4}{p}\right)
\right)- \frac{1}{6} \left(1-\left(\frac{-3}{p}\right) \right). 
  \end{split}
   \end{align}

Since there are $h_{j_1} h_{j_2}$ classes $[J]\in \Cl(\bbO_1)$ with $\bbO_J^\times
\simeq C_{2j_1}\times C_{2j_2}$, it follows from
Lemma~\ref{sec:o16-double-coset-fiber} that 
\begin{equation}
  \label{eq:hO16-1}
  h(\bbO_{16})=\sum_{1\le j_1,j_2\le 3} h_{j_1} h_{j_2} c_{j_1,j_2}.   
\end{equation}
Observe that 
\[ c_{j_1,j_2}=
\begin{cases}
  \frac{6}{j_1 j_2} & \text{if $(j_1,j_2)\neq (2,2)$ or $(j_1,j_2)\neq
  (3,3)$}; \\
  \frac{6}{j_1 j_2}+\frac{1}{2} & \text{for $(j_1,j_2)= (2,2)$;}  \\  
  \frac{6}{j_1 j_2}+\frac{4}{3} & \text{for $(j_1,j_2)= (3,3)$.} \\  
\end{cases} \]
We can express (\ref{eq:hO16-1}) as
\begin{equation}
  \label{eq:hO16}
  \begin{split}
      h(\bbO_{16}(1,2))&=\sum_{1\le j_1,j_2\le 3} h_{j_1} h_{j_2}
  \frac{6}{j_1j_2}+\frac{1}{2} h_2^2 +\frac{4}{3} h_3^2 \\
  &=6\left(h_1+\frac{h_2}{2}+\frac{h_3}{3}\right)^2+
  \frac{1}{8}\left(1-\left
  (\frac{-4}{p}\right)\right)^2+ \frac{1}{3}\left(1-\left
  (\frac{-3}{p}\right)\right)^2 \\
  &=\frac{(p-1)^2}{24}+\frac{1}{4}\left(1-\left
  (\frac{-4}{p}\right)\right)+\frac{2}{3}\left(1-\left
  (\frac{-3}{p}\right)\right).     \qedhere
  \end{split} 
\end{equation}
\end{proof}


\section*{Acknowledgments}
J.~Xue is partially supported
by the 1000-plan program for young talents and Natural Science Foundation grant \#11601395 of PRC. He thanks Academia
Sinica and NCTS for their hospitality and great working conditions.
TC Yang and CF Yu are partially supported by the MoST grants 
104-2115-M-001-001MY3, 104-2811-M-001-066, 105-2811-M-001-108 and
107-2115-M-001-001-MY2. 


\bibliographystyle{hplain}
\bibliography{TeXBiB}

\def\cprime{$'$}
\begin{thebibliography}{10}

\bibitem{Alsina-Bayer}
Montserrat Alsina and Pilar Bayer.
\newblock {\em Quaternion orders, quadratic forms, and {S}himura curves},
  volume~22 of {\em CRM Monograph Series}.
\newblock American Mathematical Society, Providence, RI, 2004.

\bibitem{Bass-MathZ-1963}
Hyman Bass.
\newblock On the ubiquity of {G}orenstein rings.
\newblock {\em Math. Z.}, 82:8--28, 1963.

\bibitem{Borel:1962}
Armand Borel.
\newblock Arithmetic properties of linear algebraic groups.
\newblock In {\em Proc. {I}nternat. {C}ongr. {M}athematicians ({S}tockholm,
  1962)}, pages 10--22. Inst. Mittag-Leffler, Djursholm, 1963.

\bibitem{MR693798}
J.~Brzezi\'nski.
\newblock On orders in quaternion algebras.
\newblock {\em Comm. Algebra}, 11(5):501--522, 1983.

\bibitem{curtis-reiner:1}
Charles~W. Curtis and Irving Reiner.
\newblock {\em Methods of representation theory. {V}ol. {I}}.
\newblock Wiley Classics Library. John Wiley \& Sons, Inc., New York, 1990.
\newblock With applications to finite groups and orders, Reprint of the 1981
  original, A Wiley-Interscience Publication.

\bibitem{Eichler-class-no1-1938}
M.~Eichler.
\newblock \"uber die {I}dealklassenzahl hyperkomplexer {S}ysteme.
\newblock {\em Math. Z.}, 43(1):481--494, 1938.

\bibitem{Farb-Dennis-1993}
Benson Farb and R.~Keith Dennis.
\newblock {\em Noncommutative algebra}, volume 144 of {\em Graduate Texts in
  Mathematics}.
\newblock Springer-Verlag, New York, 1993.

\bibitem{Heller-Reiner-1962}
A.~Heller and I.~Reiner.
\newblock Representations of cyclic groups in rings of integers. {I}.
\newblock {\em Ann. of Math. (2)}, 76:73--92, 1962.

\bibitem{Humphreys:1995}
James~E. Humphreys.
\newblock {\em Conjugacy classes in semisimple algebraic groups}, volume~43 of
  {\em Mathematical Surveys and Monographs}.
\newblock American Mathematical Society, Providence, RI, 1995.

\bibitem{karemaker-pries-2017}
V.~{Karemaker} and R.~{Pries}.
\newblock {Fully maximal and fully minimal abelian varieties}.
\newblock {\em ArXiv e-prints}, March 2017,
  \href{https://arxiv.org/abs/1703.10076}{{\ttfamily arXiv:1703.10076}}.

\bibitem{Langlands-stable-conj-1979}
R.~P. Langlands.
\newblock Stable conjugacy: definitions and lemmas.
\newblock {\em Canad. J. Math.}, 31(4):700--725, 1979.

\bibitem{li-oort}
Ke-Zheng Li and Frans Oort.
\newblock {\em Moduli of supersingular abelian varieties}, volume 1680 of {\em
  Lecture Notes in Mathematics}.
\newblock Springer-Verlag, Berlin, 1998.

\bibitem{li-xue-yu:unit-gp}
Qun {Li}, Jiangwei {Xue}, and Chia-Fu {Yu}.
\newblock {Unit groups of maximal orders in totally definite quaternion
  algebras over real quadratic fields}.
\newblock {\em ArXiv e-prints}, July 2018,
  \href{https://arxiv.org/abs/1807.04736}{{\ttfamily arXiv:1807.04736}}.

\bibitem{Pizer-1976}
Arnold Pizer.
\newblock On the arithmetic of quaternion algebras.
\newblock {\em Acta Arith.}, 31(1):61--89, 1976.

\bibitem{platonov-rapinchuk}
Vladimir Platonov and Andrei Rapinchuk.
\newblock {\em Algebraic groups and number theory}, volume 139 of {\em Pure and
  Applied Mathematics}.
\newblock Academic Press, Inc., Boston, MA, 1994.
\newblock Translated from the 1991 Russian original by Rachel Rowen.

\bibitem{pop-pop-1985}
Florian Pop and Horia Pop.
\newblock An extension of the {N}oether-{S}kolem theorem.
\newblock {\em J. Pure Appl. Algebra}, 35(3):321--328, 1985.

\bibitem{reiner:mo}
I.~Reiner.
\newblock {\em Maximal orders}, volume~28 of {\em London Mathematical Society
  Monographs. New Series}.
\newblock The Clarendon Press Oxford University Press, Oxford, 2003.
\newblock Corrected reprint of the 1975 original, With a foreword by M. J.
  Taylor.

\bibitem{Serre_local}
Jean-Pierre Serre.
\newblock {\em Local fields}, volume~67 of {\em Graduate Texts in Mathematics}.
\newblock Springer-Verlag, New York, 1979.
\newblock Translated from the French by Marvin Jay Greenberg.

\bibitem{springer-steinberg-1970}
T.~A. Springer and R.~Steinberg.
\newblock Conjugacy classes.
\newblock In {\em Seminar on {A}lgebraic {G}roups and {R}elated {F}inite
  {G}roups ({T}he {I}nstitute for {A}dvanced {S}tudy, {P}rinceton, {N}.{J}.,
  1968/69)}, Lecture Notes in Mathematics, Vol. 131, pages 167--266. Springer,
  Berlin, 1970.

\bibitem{Swan-1988}
Richard~G. Swan.
\newblock Torsion free cancellation over orders.
\newblock {\em Illinois J. Math.}, 32(3):329--360, 1988.

\bibitem{vigneras}
Marie-France Vign{\'e}ras.
\newblock {\em Arithm\'etique des alg\`ebres de quaternions}, volume 800 of
  {\em Lecture Notes in Mathematics}.
\newblock Springer, Berlin, 1980.

\bibitem{Washington-cyclotomic}
Lawrence~C. Washington.
\newblock {\em Introduction to cyclotomic fields}, volume~83 of {\em Graduate
  Texts in Mathematics}.
\newblock Springer-Verlag, New York, second edition, 1997.

\bibitem{waterhouse:thesis}
William~C. Waterhouse.
\newblock Abelian varieties over finite fields.
\newblock {\em Ann. Sci. \'Ecole Norm. Sup. (4)}, 2:521--560, 1969.

\bibitem{xue-yang-yu:ECNF}
Jiangwei {Xue}, Tse-Chung {Yang}, and Chia-Fu {Yu}.
\newblock {Supersingular abelian surfaces and Eichler's class number formula}.
\newblock {\em ArXiv e-prints}, April 2014,
  \href{https://arxiv.org/abs/1404.2978}{{\ttfamily arXiv:1404.2978}}.
\newblock to appear in {The Asian Journal of Mathematics}.

\bibitem{xue-yang-yu:num_inv}
Jiangwei Xue, Tse-Chung Yang, and Chia-Fu Yu.
\newblock Numerical invariants of totally imaginary quadratic {$\Bbb Z[\sqrt
  p]$}-orders.
\newblock {\em Taiwanese J. Math.}, 20(4):723--741, 2016.

\bibitem{xue-yang-yu:sp_as}
Jiangwei Xue, Tse-Chung Yang, and Chia-Fu Yu.
\newblock On superspecial abelian surfaces over finite fields.
\newblock {\em Doc. Math.}, 21:1607--1643, 2016.

\bibitem{xue-yu:counting}
Jiangwei {Xue} and Chia-Fu {Yu}.
\newblock {Counting abelian varieties over finite fields}.
\newblock {\em ArXiv e-prints}, January 2018,
  \href{https://arxiv.org/abs/1801.00229}{{\ttfamily arXiv:1801.00229}}.

\bibitem{xue-yu-zheng:spIII}
Jiangwei Xue, Chia-Fu Yu, and Yuqiang Zheng.
\newblock On superspecial abelian surfaces over finite fields {III}.
\newblock In preparation.

\bibitem{yu:note_ss}
C.-F. {Yu}.
\newblock {A note on supersingular abelian varieties}.
\newblock {\em ArXiv e-prints}, December 2014,
  \href{https://arxiv.org/abs/1412.7107}{{\ttfamily arXiv:1412.7107}}.

\bibitem{yu:embed}
Chia-Fu Yu.
\newblock Embeddings of fields into simple algebras: generalizations and
  applications.
\newblock {\em J. Algebra}, 368:1--20, 2012.

\bibitem{yu:sp-prime}
Chia-Fu Yu.
\newblock Superspecial abelian varieties over finite prime fields.
\newblock {\em J. Pure Appl. Algebra}, 216(6):1418--1427, 2012.

\end{thebibliography}
\end{document}